\documentclass[12pt,reqno]{amsart}
\usepackage{amssymb}
\usepackage{amscd}
\usepackage{amsmath}

\newfont{\cyr}{wncyr10} %this gives us cyrillic fonts

\theoremstyle{plain}
\newtheorem{theorem}{Theorem}[section]
\newtheorem{Thm}{Theorem}
\newtheorem{lemma}[theorem]{Lemma}
\newtheorem{proposition}[theorem]{Proposition}
\newtheorem{corollary}[theorem]{Corollary}

\newtheorem*{proposition*}{Proposition}

\theoremstyle{definition}
\newtheorem{remark}[theorem]{Remark}
\newtheorem{definition}[theorem]{Definition}

\numberwithin{equation}{section}

\newcommand{\ab}{\mathrm {ab}}

\newcommand{\cC}{{\mathcal C}}

\newcommand{\bD}{{\mathbf D}}

\newcommand{\cE}{{\mathcal E}}

\newcommand{\ff}{{\mathfrak f}}

\newcommand{\bG}{{\mathbf G}}

\newcommand{\ch}[1]{{\check{#1}}}

\newcommand{\wh}[1]{{\widehat{#1}}}

\newcommand{\cI}{{\mathcal I}}

\newcommand{\cK}{{\mathcal K}}
\newcommand{\fK}{{\mathfrak K}}
\renewcommand{\L}{{\mathcal L}}
\newcommand{\cL}{{\mathcal L}}

\newcommand{\fm}{{\mathfrak m}}
\newcommand{\ms}{{\Sigma}}  %notation for restricted Selmer groups
\newcommand{\bN}{{\mathbf N}}
\newcommand{\cN}{{\mathcal N}}
\newcommand{\cO}{{\mathcal O}}
\newcommand{\ov}[1]{{\overline{#1}}}

\newcommand{\fp}{{\mathfrak p}}

\newcommand{\bQ}{{\mathbf Q}}
\newcommand{\fq}{{\mathfrak q}}

\newcommand{\rel}{\mathrm {rel}}

\newcommand{\sha}{{\mbox{\cyr Sh}}}

\newcommand{\vs}{{\varsigma}}

\newcommand{\vt}{{\vartheta}}

\newcommand{\cY}{{\mathcal Y}}
\newcommand{\bZ}{{\mathbf Z}}

\DeclareMathOperator{\Aut}{Aut}

\DeclareMathOperator{\Fr}{Fr}
\DeclareMathOperator{\Gal}{Gal}

\DeclareMathOperator{\Hom}{Hom}
\DeclareMathOperator{\Image}{Im}
\DeclareMathOperator{\Ker}{Ker}

\DeclareMathOperator{\loc}{\mathrm{loc}}

\DeclareMathOperator{\ord}{ord}

\DeclareMathOperator{\rank}{rank}

\DeclareMathOperator{\Sel}{Sel}

\DeclareMathOperator{\Tr}{Tr}

\begin{document}
\title[Birch and Swinnerton-Dyer conjecture II]{On Rubin's variant
of the $p$-adic Birch and\\ Swinnerton-Dyer conjecture II}
\author{A. Agboola}
\date{Final version. April 1, 2010}
\address{Department of Mathematics \\
University of California \\ Santa Barbara, CA 93106.}
\email{agboola@math.ucsb.edu}
\thanks{Partially supported by NSA grant no. H98230-08-1-0077. I am very
  grateful to the anonymous referee for a number of very helpful
  comments. This paper was completed while I was visiting the Humboldt
  University, Berlin. I would like to thank the members of the
  Mathematics Department there for their generous hospitality.}
\begin{abstract}
Let $E/\bQ$ be an elliptic curve with complex multiplication by the
ring of integers of an imaginary quadratic field $K$. In 1991, by
studying a certain special value of the Katz two-variable $p$-adic
$L$-function lying outside the range of $p$-adic interpolation,
K. Rubin formulated a $p$-adic variant of the Birch and
Swinnerton-Dyer conjecture when $E(K)$ is infinite, and he proved that
his conjecture is true for $E(K)$ of rank one.

When $E(K)$ is finite, however, the statement of Rubin's original
conjecture no longer applies, and the relevant special value of the
appropriate $p$-adic $L$-function is equal to zero. In this paper we
extend our earlier work and give an unconditional proof of an analogue
of Rubin's conjecture when $E(K)$ is finite.
\end{abstract}

\maketitle

\section{Introduction} \label{S:intro}

The goal of this article is to extend the results of \cite{A1} to give
an unconditional proof of a certain variant of the $p$-adic Birch and
Swinnerton-Dyer conjecture for elliptic curves with complex
multiplication.

Let $E/\bQ$ be an elliptic curve with complex multiplication by $O_K$,
the ring of integers of an imaginary quadratic field $K$; this implies
that $K$ is necessarily of class number one. Let $p>3$ be a prime of
good, ordinary reduction for $E$; then we may write $pO_K=\fp \fp^*$,
with $\fp = \pi O_K$ and $\fp^* = \pi^* O_K$.

Set $\cK_{\infty}:= K(E_{\pi^{\infty}})$, $\cK_{\infty}^{*}:=
K(E_{\pi^{*\infty}})$, and $\fK_{\infty}:= \cK_{\infty}
\cK_{\infty}^{*}$. Let $\cO$ denote the completion of the ring of
integers of $\cK_{\infty,\fp}^{*}$. For any extension $L/K$ we set
$\Lambda(L):= \Lambda(\Gal(L/K)):= \bZ_p[[\Gal(L/K)]]$, and
$\Lambda(L)_{\cO}:= \cO[[\Gal(L/K)]]$.

Let
\begin{align*}
&\psi: \Gal(\ov{K}/K) \to \Aut(E_{\pi^{\infty}}) \xrightarrow{\sim}
O_{K,\fp}^{\times} \xrightarrow{\sim} \bZ_{p}^{\times}, \\
&\psi^*: \Gal(\ov{K}/K) \to \Aut(E_{\pi^{*\infty}}) \xrightarrow{\sim}
O_{K,\fp^*}^{\times} \xrightarrow{\sim} \bZ_{p}^{\times}
\end{align*}
denote the natural $\bZ_{p}^{\times}$-valued characters of
$\Gal(\ov{K}/K)$ arising via Galois action on $E_{\pi^\infty}$ and
$E_{\pi^{*\infty}}$ respectively. We may identify $\psi$ with the
Grossecharacter associated to $E$ (and $\psi^*$ with the complex
conjugate $\ov{\psi}$ of this Grossencharacter), as described, for
example, in \cite[p. 325]{R}. We write $T$ and $T^*$ for the
$\fp$-adic and $\fp^*$-adic Tate modules of $E$ respectively.

We now recall that the Katz two-variable $p$-adic $L$-function
$\cL_{\fp} \in \Lambda(\fK_{\infty})_{\cO}$ satisfies a $p$-adic
interpolation formula that may be described as follows (see
\cite[Theorem 7.1]{R} for the version given here, and also
\cite[Theorem II.4.14]{dS}. Note also that, as the notation
indeicates, $\cL_{\fp}$ depends upon a choice of prime $\fp$ lying
above $p$---cf. Remark \ref{R:switch} below).  For all pairs of
integers $j,k \in \bZ$ with $0 \leq -j <k$, and for all characters
$\chi: \Gal(K(E_p)/K) \to \ov{K}^{\times}$, we have
\begin{equation} \label{E:interpolate}
\L_{\fp}(\psi^k \psi^{*j} \chi) = A \cdot L(\psi^{-k} \ov{\psi}^{-j}
\chi^{-1}, 0).
\end{equation}
Here $L(\psi^{-k} \ov{\psi}^{-j} \chi^{-1},s)$ denotes the complex
Hecke $L$-function, and $A$ denotes an explicit, non-zero factor whose
precise description we shall not need.

Write $\langle \psi \rangle : \Gal(\ov{K}/K) \to 1 + p\bZ_p$ for the
composition of $\psi$ with the natural projection $\bZ_{p}^{\times}
\to 1+p\bZ_p$, and define $\langle \psi^* \rangle$ in a similar
manner. Define
$$
L_{\fp}(s):= \L_{\fp}(\psi \langle \psi \rangle^{s-1}),\quad L_{\fp}^{*}(s):=
\L_{\fp}(\psi^* \langle \psi^{*}\rangle^{s-1})
$$
for $s \in \bZ_p$. The character $\psi$ lies within the range of
interpolation of $\cL_{\fp}$, and the behaviour of $\cL_{\fp}$ at
$\psi$ is predicted by the $\fp$-adic Birch and
Swinnerton-Dyer conjecture for $E$ (see \cite[pages 133--134]{BGS},
\cite[Theorem V.8]{PR7}). Conjecturally, $\ord_{s=1} L_{\fp}(s)$ is
equal to the rank $r$ of $E(\bQ)$, and the exact value of
$\lim_{s \to 1}L_{\fp}(s)/[s-1]^r$ may be described in terms of various
arithmetic invariants associated to $E$.

The character $\psi^*$, however, lies outside the range of
interpolation of $\cL_{\fp}$, and the function $L_{\fp}^*(s)$ has not
  been studied nearly as much as $L_{\fp}(s)$. The first results
  concerning the behaviour of $L_{\fp}^*(s)$ were obtained by Karl
  Rubin (see \cite{R}, \cite{R1}). When $r \geq 1$, Rubin formulated a
  variant of the $\fp$-adic Birch and Swinnerton-Dyer conjecture for
  $L_\fp^*(s)$ which predicts that that $\ord_{s=1} L_{\fp}^*(s)$ is
  equal to $r-1$, and which gives a formula for $\lim_{s \to
    1}[L_{\fp}^*(s)/(s-1)^{r-1}]$. Under suitable hypotheses, Rubin
  showed that his conjecture is equivalent to the usual $\fp$-adic
  Birch and Swinnerton-Dyer conjecture, and he proved both conjectures
  when $r=1$. In the case $r=1$, he then used these results to give
  the first examples of a $p$-adic construction of a global point of
  infinite order in $E(\bQ)$ directly from the special value of a
  $p$-adic $L$-function.

On the other hand, when $r=0$, matters become quite different, and
the results of \cite{R}, \cite{R1} do not apply. It is not hard to
show that the functional equation satisfied by $\cL_{\fp}$ (see
\cite[\S6]{dS}) implies that $\ord_{s=1} L_{\fp}^*(s) \geq 1$. Rubin
speculated that perhaps $\ord_{s=1} L_{\fp}^*(s)=1$, and assuming that
this was true, he also raised the question of determining the value of
$\lim_{s \to 1}[L_\fp^*(s)/(s-1)]$ (see \cite[Remark on p.74]{R1}).

A natural approach to studying the behaviour of $\cL_{\fp}$ at
$\psi^*$ is to try to find a suitable Selmer group that governs the
behaviour of $L_{\fp}^{*}(s)$ at $s=1$. This may be done by using the
two-variable main conjecture to study certain Iwasawa modules
associated naturally to $L_{\fp}^{*}(s)$, as described in detail in
\cite[\S1--\S3]{A1}.  In \cite{A1}, we defined a \textit{restricted
  Selmer group} $\ch{\ms}_{\fp^*}(T^*) \subseteq H^1(K,T^*)$. This
restricted Selmer group is defined by \textit{reversing} the Selmer
conditions above $\fp$ and $\fp^*$ that are used to define the usual
Selmer group $\Sel(K,T^*)$. The $O_{K,{\fp^*}}$-module
$\ch{\ms}_{\fp^*}(K,T^*)$ is free of rank $|r-1|$, and if $r \geq 1$,
then in fact $\ch{\ms}_{\fp^*}(K,T^*) \subseteq \Sel(K,T^*)$ (see
\cite[Lemma 3.6, Proposition 6.5, and Proposition 6.7]{A1}). We also
defined a similar group $\ch{\ms}_{\fp}(K,T) \subseteq H^1(K,T)$, and
we constructed a $p$-adic height pairing
$$
[\,,\,]_{K,{\fp^*}}: \ch{\ms}_{\fp}(K,T) \times
\ch{\ms}_{\fp^*}(K,T^*) \to O_{K,{\fp^*}}.
$$
If $r \geq 1$, and if the $\fp^*$-adic Birch and Swinnerton-Dyer
conjecture is true, then $[\,,\,]_{K,\fp^*}$ is non-degenerate. We
conjectured that $[\,,\,]_{K,\fp^*}$ is also non-degenerate when
$r=0$.

It was shown in \cite{A1} that if $[\,,\,]_{K,\fp^*}$ is
non-degenerate and the $p$-primary part of $\sha(E/K)$ is finite, then
\begin{equation*}
\ord_{s= 1} L_\fp^*(s) = \rank_{O_{K,{\fp^*}}}(\ch{\ms}_{\fp^*}(K,T^*))
= |r-1|.
\end{equation*}
Under these assumptions, we also determined the value of $\lim_{s
\to 1} L_\fp^*(s)/(s-1)^{|r-1|}$ up to multiplication by a $p$-adic
unit (thereby recovering a weak form of \cite[Corollary 11.3]{R} in
the case $r \geq 1$). In particular, our results implied that if $r=0$
(in which case $\sha(E/K)$ is known to be finite; see \cite{Ru}) and
$[\,,\,]_{K,\fp^*}$ is non-degenerate, then $\ord_{s=1}L_\fp^*(s) =1$,
as was guessed by Rubin in \cite{R1}.

Suppose now that $r=0$. In this paper, we strengthen the results of
\cite{A1} by giving an unconditional proof of the fact that
$\ord_{s=1}(L_\fp^*(s)) = 1$, and we also determine the exact value of
the first derivative of $L_{\fp}^{*}(s)$ at $s=1$. We do this via an
approach involving elliptic units and explicit reciprocity laws
(cf. \cite{R1}), rather than the two-variable main conjecture and
Galois cohomology, as in \cite{A1}. In order to state our main
result, we must introduce some further notation.

Suppose that $ y \in \ch{\ms}_{\fp}(K,T)$ and $y^* \in
\ch{\ms}_{\fp^*}(K,T^*)$ are of infinite order, and let
\begin{equation*}
\exp_\fp^*: H^1(K_\fp, T^*) \to \bQ_p,\qquad \exp_{\fp^*}^{*}:
H^1(K_{\fp^*},T) \to \bQ_p
\end{equation*}
denote the Bloch-Kato dual exponential maps. Via localisation, these
induce maps (which we denote by the same symbols):
$$
\exp_\fp^*: \ch{\ms}_{\fp^*}(K,T^*) \to \bQ_p,\qquad \exp_{\fp^*}^{*}:
\ch{\ms}_\fp(K,T) \to \bQ_p.
$$
Write $\ff \subseteq O_K$ for the conductor of the Grossencharacter
associated to $E$, and let $\bN(\ff)$ denote the norm of this ideal.
Set
$$
\cL_{\fp}'(\psi^*) := \lim_{s \to 1} \frac{L_\fp^*(s)}{s-1}.
$$

The following result may perhaps be viewed as being an analogue of a
similar exceptional zero phenomenon observed in the work of Mazur,
Tate and Teitelbaum concerning $p$-adic Birch and Swinnerton-Dyer
conjectures for elliptic curves \textit{without} complex
multiplication (see \cite[especially Conjecture 2]{Gr94},
\cite[especially page 38]{MTT}). It relates the \textit{value} of
$\cL_{\fp}$ at a point within the range of $p$-adic interpolation to
the \textit{derivative} of $\cL_{\fp}$ at a point lying outside the
range of $p$-adic interpolation.

\begin{Thm} \label{T:A}
Suppose that  $\cL_\fp(\psi) \neq 0$. 

(a) The $p$-adic height pairing $[\,,\,]_{K,\fp^*}$ is non-degenerate
  and $\ord_{s=1}L_\fp^*(s) = 1$.

(b) We have that 
\begin{align*}
(p-1) \cdot &\left( 1-\frac{1}{\psi(\fp^*)} \right) \cdot
\left( 1-\frac{1}{\psi^*(\fp)} \right) \cdot \frac{\Omega_{\fp} \cdot
\cL_{\fp}'(\psi^*)}{[y, y^*]_{K,\fp^*}} = \\
&p \cdot \bN(\ff) \cdot \left( 1 - \frac{\psi^*(\fp)}{p} \right) \cdot \left(1 -
\frac{\psi(\fp^*)}{p} \right) \cdot
\frac{\cL_{\fp}(\psi)}{\Omega_{\fp^*} \cdot 
\exp^{*}_{\fp}(y^*) \cdot \exp^{*}_{\fp^*}(y)},
\end{align*}

%\begin{align*}
%& (p-1) \cdot \left( 1 - \frac{\psi^*(\fp)}{p} \right)^{-1} \cdot \left(1 -
%  \frac{\psi(\fp^*)}{p} \right)^{-1} \cdot \frac{\Omega_{\fp} \cdot
%    \cL_{\fp}'(\psi^*)}{[y, y^*]_{K,\fp^*}} = \\
%&\bN(\ff) \cdot (1-\psi(\fp^*)^{-1})^{-1} \cdot
%  (1-\psi^*(\fp)^{-1})^{-1} \cdot
%  \frac{\cL_{\fp^*}(\psi^*)}{\Omega_{\fp^*} \cdot
%    \exp_{\fp}(y^*) \cdot \exp_{\fp^*}(y)},
%\end{align*}
where $\Omega_\fp$ and $\Omega_{\fp^*}$ are certain $p$-adic
periods defined using the formal group associated to $E$ (see Section
\ref{S:formal} below).
\end{Thm}

It is interesting to compare Theorem \ref{T:A} with \cite[Theorem
  10.1]{R}. Both of these results are examples of the following more
general phenomenon concerning certain special values of the Katz
two-variable $p$-adic $L$-function $\cL_\fp$. Let $k \geq 0$ be an
integer, and set
$$
\phi_k:= \psi^{k+1} \psi^{*-k},\qquad \phi_k^*:= \psi^{-k} \psi^{*k+1}.
$$ Then we see from \eqref{E:interpolate} that $\phi_k$ lies within
the range of interpolation of $\cL_\fp$; for $k \geq 1$, the behaviour
of $\cL_\fp$ at $\phi_k$ is predicted by various conjectures due to
Bloch, Beilinson, Kato and Perrin-Riou. On the other hand, the
character $\phi_k^*$ lies outside the range of interpolation of
$\cL_\fp$, and as far the present author is aware, the behaviour of
$\cL_\fp$ at $\phi_k^*$ for $k \geq 1$ does not appear to have
previously been studied. Write $\langle \phi_k \rangle$ (respectively
$\langle \phi_k^* \rangle$) for the composition of $\phi_k$
(respectively $\phi_k^*$) with the natural projection
$\bZ_{p}^{\times} \to 1+p\bZ_p$, and define
$$
L_{\fp}(\phi_k,s):= \cL_{\fp}(\phi_k \langle \phi_k \rangle^{s-1}),
\qquad
L_{\fp}(\phi^*_k,s):= \cL_{\fp}(\phi^*_k \langle \phi^*_k
\rangle^{s-1})
$$
for $s \in \bZ_p$.

Using the techniques of \cite{R}, \cite{A1} and the present paper, it
may be shown that the orders of vanishing of $L_{\fp}(\phi_k,s)$ and
$L_{\fp}(\phi_k^*,s)$ at $s=1$ are of opposite parity, and that the
values of the first non-vanishing derivatives of $L_{\fp}(\phi_k,s)$
and $L_{\fp}(\phi_k^*,s)$ at $s=1$ are related in a manner similar to
\cite[Theorem 10.1]{R} (if $\cL_\fp(\phi_k) =0$) or to Theorem
\ref{T:A} above (if $\cL_\fp(\phi_k) \neq 0$). We shall discuss this
more fully in a future article (see \cite{A2}).

The strategy of the proof of Theorem \ref{T:A} is similar to that
employed in \cite{R}; however, because we work with restricted Selmer
groups rather than true Selmer groups, the details are rather
different. These differences mainly arise from the fact that the
$p$-adic height pairing $[\,,\,]_{K,\fp^*}$ on restricted Selmer
groups (when $r=0$) is somewhat more difficult to work with than is the
$p$-adic height pairing on true Selmer groups (when $r=1)$. The basic
ideas involved in the proof of Theorem \ref{T:A} may be described as
follows. Using elliptic units, we construct canonical elements
$$
s_\fp \in \ch{\ms}_\fp(K,T), \qquad s_{\fp^*} \in
\ch{\ms}_{\fp^*}(K,T^*).
$$ 
It follows from the proof of Theorem \ref{T:nondeg}(a) below that
$s_{\fp^*}$ is of infinite order only if $\cL'_{\fp}(\psi^*) \neq
0$. By analysing certain Kummer and cup product pairings, and using
Wiles's explicit reciprocity for formal groups, we prove that
\begin{equation} \label{E:sinf}
\exp_\fp^*(s_{\fp^*}) \doteq \cL_\fp(\psi),
\end{equation}
where the symbol ``$\doteq$'' denotes equality up to multiplication by
a non-zero factor. Hence we see that if $\cL_{\fp}(\psi) \neq
0$, then $s_{\fp^*}$ is indeed of infinite order, and so $L_{\fp}^{*}(s)$ has
a first-order zero at $s=1$.  

We then compute $[s_\fp,s_{\fp^*}]_{K,{\fp^*}}$ using Kummer theory,
Hilbert symbols, and Wiles's explicit reciprocity law, and we see that
\begin{equation} \label{E:sht}
[s_\fp,s_{\fp^*}]_{K,{\fp^*}} \doteq \cL_\fp(\psi) \cdot
\cL_{\fp}'(\psi^*)
\end{equation}
The proof of Theorem \ref{T:A} is then completed by showing that if $y
\in \ch{\ms}_{\fp}(K,T)$ and $y^* \in \ch{\ms}_{\fp^*}(K,T^*)$ are
  both of infinite order, then
\begin{equation*}
\frac{\exp_{\fp}^{*}(y^*) \cdot \exp_{\fp^*}^*(y)}{[y,
    y^*]_{K, {\fp^*}}} =
\frac{\exp_{\fp}^{*}(s_{\fp^*}) \cdot \exp_{\fp^*}^*(s_\fp)}{[s_\fp,
    s_{\fp^*}]_{K, {\fp^*}}},
\end{equation*}
and so the desired result follows by applying \eqref{E:sinf} and
\eqref{E:sht}.

An outline of the contents of this paper is as follows. In Section
\ref{S:selmer} we recall some basic properties of restricted Selmer
groups. In Section \ref{S:height} we describe a local decomposition of
the $p$-adic height pairing $[\,,\,]_{K,\fp^*}$ in terms of local
Artin symbols. We recall some basic facts concerning the formal group
$\hat{E}$ associated to $E$ in Section \ref{S:formal}, and we
establish a number of conventions for use in subsequent
calculations. We discuss properties of various Kummer pairings in
Section \ref{S:kummer}, and we use these pairings to compute the value
of $[\,,\,]_{K,\fp^*}$ on certain cohomology classes in restricted
Selmer groups that are constructed using global units. In Section
\ref{S:dual}, we use the results of Section \ref{S:kummer} to compute
certain special values of the dual exponential map in terms of
logarithmic derivatives of certain Coleman power series. In Section
\ref{S:elliptic}, we describe the construction of the canonical
elements $s_{\fp}$ and $s_{\fp^*}$ via elliptic units. Finally, in
Section \ref{S:special}, we apply our previous results to these
canonical elements, and we prove Theorem \ref{T:A}.

We conclude this Introduction by remarking that the methods and
results of this paper remain valid if we assume that $E$ is defined
over its field of complex multiplication $K$ rather than over $\bQ$.
\medskip

%%%%%%%%%%%%%%%%%%%%%%%%%%%%%%%%%%%%%%%%%%%%%%%%%%%%%%%%%%%%%%%%%%%%%%%%

\noindent{}{\bf Notation and conventions.} Throughout this paper, $K$
denotes an imaginary quadratic field of class number one. If $L$ is
any field, we write $L^{\ab}$ for the maximal abelian extension of
$L$, and $\ov{L}$ for an algebraic closure of $L$.

For each integer $n \geq 1$, we write
$$
\cK_n:= K(E_{\pi^n}),\quad \cK_n^*:= K(E_{\pi^{*n}}), \quad \fK_n:=
K(E_{p^n}) = \cK_n \cdot \cK_n^*,
$$
and 
$$
\cK_\infty:= K(E_{\pi^\infty}), \quad \cK_\infty^*:=
K(E_{\pi^{*\infty}}), \quad \fK_\infty:= K(E_{p^\infty}).
$$ 
We also put $\cN_n:= \cK_n \cdot \cK_{\infty,}^{*}$, and we write
$\fm_{n,\fp}$ for the maximal ideal of the completion of the ring of
integers of $\cN_{n,\fp}$. The symbol $\cO$ denotes the completion of
the ring of integers of $\cK_{\infty,\fp}^{*}$.

For any extension $L/K$ we set $\Lambda(L):= \Lambda(\Gal(L/K)):=
\bZ_p[[\Gal(L/K)]]$, and $\Lambda(L)_{\cO}:= \cO[[\Gal(L/K)]]$. We
write
$$
\cI^*:= \Ker(\psi^*: \Lambda(\cK^*_\infty) \to \bZ_p), \quad
\cI:= \Ker(\psi: \Lambda(\cK_\infty) \to \bZ_p),
$$ and let $\vt^*$ and $\vt$ be the generators of $\cI^*$ and $\cI$
fixed in \cite[\S6]{R}; so $\vt^* = \gamma \psi^*(\gamma^{-1}) -1$,
where $\gamma$ is any topological generator of $\Gal(\cK_\infty^*/K)$
satisfying $\log_p(\psi^*(\gamma)) = p$, and $\vt$ is defined
analogously. 

We set $D_{\fp}:= K_{\fp}/O_{K,\fp}$ and $D_{\fp^*}:=
K_{\fp^*}/O_{K,\fp^*}$. If $M$ is any $\bZ_p$-module, we write
$M^{\land}$ for the Pontryagin dual of $M$.

For each integer $n \geq 1$, we let
$$
e_n: E_{\pi^n} \times E_{\pi^{*n}} \to \mu_{p^n}
$$ 
denote the Weil pairing, normalised as decribed in
\cite[\S3.1.2]{PR1} (the reader should note that this is \textit{not}
the same normalisation as that used in \cite{R}). This pairing
satisfies the identities
\begin{equation*}
e_n(\pi^* \vs_n, \vs_n^*) = e_n(\vs_n, \pi \vs_n)
\end{equation*}
for $\vs_n \in E_{\pi^n}$, $\vs_n^* \in E_{\pi^{*n}}$, and
\begin{equation*}
e_{n+1}(\vs_n, \vs_{n+1}^{*}) = e_n(\vs_n, \pi^*\vs^{*}_{n+1})
\end{equation*}
for $\vs_n \in E_{\pi^n},\quad \vs_{n+1}^{*} \in E_{\pi^{*(n+1)}}$.

We set $W:= E_{\pi^\infty}$ and $W^*:= E_{\pi^{*\infty}}$. We write
$T$ and $T^*$ for the $\pi$-adic and $\pi^*$-adic Tate modules of $E$
respectively. Let $w=[w_n]$ and $w^*=[w_n^*]$ denote generators of $T$
and $T^*$ respectively. 

%For the rest of the paper we fix generators $w = (w_n)_{n \geq 1} \in
%T$ and $w^* =(w_n^*)_{n \geq 1} \in T^*$ of $T$ and $T^*$
%respectively, according to the recipe described in \cite[\S6]{R1}. We
%set $\zeta_n:= e_n(w_n,w_n^*)$; then $\zeta:= (\zeta_n)_{n \geq 1} \in
%\bZ_p(1)$ is a fixed generator of $\bZ_p(1)$.

We use the following notation to denote various unit groups:
\begin{align*}
%&\fU_{n,\fp}:= \text{units in $\fK_{n,\fp}$ congruent to $1$ modulo
%    $\fp$;} \\
%&\fU_{n,\fp^*}:= \text{units in $\fK_{n,\fp^*}$ congruent to $1$ modulo
%   $\fp^*$;} \\
&U_{n,\fp}:= \text{units in $\cK_{n,\fp}$ congruent to $1$ modulo
    $\fp$;} \\
&U_{n,\fp^*}:= \text{units in $\cK_{n,\fp^*}$ congruent to $1$ modulo
    $\fp^*$;} \\
&U_{\infty,\fp}:= \varprojlim U_{n,\fp},\quad U_{\infty,\fp^*}:= 
\varprojlim U_{n,\fp^*};\\
&U^{*}_{n,\fp}:= \text{units in $\cK^{*}_{n,\fp}$ congruent to $1$ modulo
    $\fp$;} \\
&U^{*}_{n,\fp^*}:= \text{units in $\cK^{*}_{n,\fp^*}$ congruent to $1$ modulo
    $\fp^*$;} \\
&U^{*}_{\infty,\fp}:= \varprojlim U^{*}_{n,\fp},\quad U^{*}_{\infty,\fp^*}:=
    \varprojlim U_{n,\fp^*},
\end{align*}
where all inverse limits are taken with respect to the obvious norm
maps. We also set
\begin{align*}
&\cE_n:= \text{global units of $\cK_n$},\quad \cE_n^*:= \text{global
  units of $\cK_n^*$};\\
&\ov{\cE}_n:= \text{the closure of the projection of $\cE_n$ into
  $U_{n,\fp}$}; \\
&\ov{\cE}^*_n:= \text{the closure of the projection of $\cE^*_n$ into
  $U^*_{n,\fp^*}$}; \\
&\ov{\cE}_{\infty}:= \varprojlim \ov{\cE}_n, \quad \ov{\cE}^{*}_{\infty}:=
  \varprojlim \ov{\cE}^*_n.
\end{align*}

\begin{remark} \label{R:leo}
Note that since the strong Leopoldt conjecture holds for all
abelian extensions of $K$ (see \cite{B}), we have that 
$$
\ov{\cE}_n \simeq \cE_n \otimes_{\bZ} \bZ_p, \quad 
\ov{\cE}^*_n \simeq {\cE}^*_n \otimes_{\bZ} \bZ_p,
$$ 
and so we may also view $\ov{\cE}_{\infty}$ as being a submodule of
$U_{\infty,\fp^*}$ and $\ov{\cE}^{*}_{\infty}$ as being a submodule of
$U^{*}_{\infty,\fp}$. We shall do this without further comment several
times in what follows. \qed
\end{remark}

\begin{remark} \label{R:switch} It is important for the reader to bear
in mind that every theorem or construction in this paper that depends
upon a choice of prime $\fp$ of $K$ lying above $p$ also has a
corresponding version in which the roles of $\fp$ and $\fp^*$ are
interchanged. We shall sometimes make use of this fact without stating
it explicitly. \qed 
\end{remark}

%%%%%%%%%%%%%%%%%%%%%%%%%%%%%%%%%%%%%%%%%%%%%%%%%%%%%%%%%%%%%%%%%%%%%%%
%%%%%%%%%%%%%%%%%%%%%%%%%%%%%%%%%%%%%%%%%%%%%%%%%%%%%%%%%%%%%%%%%%%%%%%

\section{Restricted Selmer groups} \label{S:selmer}

In this section we shall recall some basic properties of restricted
Selmer groups. We refer the reader to \cite[Sections 3 and 4]{A1} for
more complete details.

Suppose that $F/K$ is any finite extension. For any place $v$ of $F$,
we define $H^1_f(F_v,W)$ to be the image of $E(F_v) \otimes D_\fp$
under the Kummer map
$$
E(F_v) \otimes_{O_K} D_\fp \to H^1(F_v,W),
$$
and we define $H^1_f(F_v,W^*)$ in a similar manner. Note that
$H^1_f(F_v,W)=0$ if $v\nmid \fp$. We also set
\begin{align*}
&H^1_f(F_v,E_{\pi^n}):= \Image[E(F_v)/\pi^nE(F_v) \to
H^1(F_v,E_{\pi^n})],\\
&H^1_f(F_v,E_{\pi^{*n}}):= \Image[E(F_v)/\pi^{*n}E(F_v) \to
H^1(F_v,E_{\pi^{*n}})].
\end{align*}

Suppose that $M \in \{W,W^*,E_{\pi^n}, E_{\pi^{*n}}\}$ and that
$\fq \in \{\fp,\fp^*\}$ . If $c \in H^1(F,M)$, then we write
$\loc_v(c)$ for the image of $c$ in $H^1(F_v,M)$. We define

$\bullet$ the {\it true Selmer group} $\Sel(F,M)$ by
$$
\Sel(F,M) = \left\{ c \in H^1(F,M) \mid \loc_v(c) \in
H^1_f(F_v,M)\, \text{for all $v$} \right\};
$$

$\bullet$ the {\it relaxed Selmer group} $\Sel_{\rel}(F,M)$ by
$$
\Sel_{\rel}(F,M) = \left\{ c \in H^1(F,M) \mid \loc_v(c) \in
H^1_f(F_v,M)\, \text{for all $v$ not dividing $p$}
\right\};
$$

$\bullet$ the {\it $\fq$-restricted Selmer group} (or simply {\it
restricted Selmer group} for short when $\fq$ is understood)
$\ms_\fq(F,M)$ by
$$
\ms_\fq(F,M) = \left\{ c \in \Sel_{\rel}(F,M) \mid \loc_v(c) = 0\,
\text{for all $v$ dividing $\fq$} \right\}.
$$
(The terminology `restricted Selmer group' is meant to reflect a
choice of a combination of relaxed and strict Selmer conditions at
places above $p$.)

We also define
$$
\ch{\ms}_{\fq}(F,T):= \varprojlim_{n} \ms_{\fq}(F,E_{\pi^n}),\quad
\ch{\ms}_{\fq}(F,T^*):= \varprojlim_{n} \ms_{\fq}(F,E_{\pi^{*n}}).
$$

If $L/K$ is an infinite extension, we define
$$
\ms_{\fq}(L,M) = \varinjlim \ms_{\fq}(L',M),
$$
where the direct limits are taken with respect to restriction over all
subfields $L' \subset L$ finite over $K$.

%For any extension $L/K$, we set
%$$
%\ms_\fq(L,M)^{\land} = X_\fq(L,M).
%$$

We record the following standard cohomological result that will be
used later.

\begin{lemma} \label{L:resinj} Let $n \geq 0$ be an integer, and
  suppose that $L$ and $M$ are fields with $K \subseteq L \subseteq M
  \subseteq \cN_n$. Then, for every integer $m \geq 1$, the
  restriction maps
$$ 
H^1(L, E_{\pi^m}) \to H^1(M,E_{\pi^m}),\qquad H^1(L, \mu_{p^m}) \to
H^1(M, \mu_{p^m})
$$
are injective and they induce isomorphisms
$$
H^1(L,E_{\pi^m}) \simeq H^1(M,E_{\pi^m})^{\Gal(M/L)},\qquad 
H^1(L,\mu_{p^m}) \simeq H^1(M, \mu_{p^m})^{\Gal(M/L)}.
$$ 

A similar result holds if $L$ and $M$ are replaced by $L_{\fp}$ and
$M_{\fp}$ with $K_{\fp} \subseteq L_{\fp} \subseteq M_{\fp}
\subseteq \cN_{n,\fp}$.
\end{lemma}

\begin{proof} This is quite standard, and may be proved via the
  argument given in \cite[p.40]{PR7}, for example.
\end{proof}

We now explain how elements in restricted Selmer groups may be
constructed by combining Kummer theory on $E$ with Kummer theory on
the multiplicative group (see Proposition \ref{P:kummerinj} below). In
order to do this, we require several preparatory lemmas. Our starting
point is the following result of Perrin-Riou.

\begin{lemma} \label{L:kummer1}
There is an $O_K$-linear isomorphism of $\Gal(\cK^*_n/K)$-modules
\begin{equation} \label{E:mystiso}
H^1(\cK^*_n,E_{\pi^{n}}) \xrightarrow{\sim}
\Hom(E_{\pi^{*n}},\cK_n^{*\times}/\cK_{n}^{*\times p^n});\quad f \mapsto
\tilde{f}.
\end{equation}
For each place $v$ of $\cK^*_n$, there is also a corresponding local
$O_K$-linear isomorphism
$$
H^1(\cK^{*}_{n,v},E_{\pi^{n}}) \xrightarrow{\sim}
\Hom(E_{\pi^{*n}},\cK_{n,v}^{*\times}/\cK_{n,v}^{*\times p^n}).
$$
\end{lemma}

\begin{proof} See \cite[Lemme 3.8]{PR1}. The isomorphism
\eqref{E:mystiso} is defined as follows. Let $f \in
H^1(\cK_n^*,E_{\pi^{n}})$, and recall that
$$
e_n: E_{\pi^n} \times E_{\pi^{*n}} \to \mu_{p^n}
$$ 
denotes the Weil pairing. We identify $\cK_n^{*\times}/\cK_{n}^{*\times
p^n}$ with $H^1(\cK_n^*, \mu_{p^n})$ via Kummer theory. If $\vs^* \in
E_{\pi^{*n}}$, then $\tilde{f}(\vs^*) \in H^1(\cK_n^*, \mu_{p^n})$ is
defined to be the element represented by the cocycle
$$
\sigma \mapsto e_n(f(\sigma),\vs^*)
$$
for all $\sigma \in \Gal(\ov{K}/\cK_n^*)$.
\end{proof}

\begin{corollary} \label{C:riso}
There are isomorphisms
\begin{align*}
&r_n: H^1(K,E_{\pi^n}) \xrightarrow{\sim} \Hom(E_{\pi^{*n}},
\cK_{n}^{*\times}/\cK_{n}^{*\times p^{n}})^{\Gal(\cK_n^*/K)} \\
&r^*_n: H^1(K,E_{\pi^{*n}}) \xrightarrow{\sim} \Hom(E_{\pi^{n}},
\cK_{n}^{\times}/\cK_{n}^{\times p^{n}})^{\Gal(\cK_n/K)}.
\end{align*}
\end{corollary}

\begin{proof} This follows from Lemmas \ref{L:resinj} and
  \ref{L:kummer1} (cf. also \cite[Lemma 2.1]{R} or \cite[Lemme
    3.8]{PR1}).
\end{proof}

\begin{lemma} \label{L:kummer2}
For each place $v$ of $\cK^*_n$ with $v \nmid \fp^*$, there is
an $O_K$-linear isomorphism
$$
E(\cK^{*}_{n,v})/\pi^{n}E(\cK^{*}_{n,v}) \xrightarrow{\sim}
\Hom(E_{\pi{*^n}}, O_{\cK^{*}_{n,v}}^{\times}/O_{\cK^{*}_{n,v}}^{\times p^n}).
$$
\end{lemma}

\begin{proof} See \cite[Lemme 3.11]{PR1}.
\end{proof}

\begin{corollary} \label{C:kummer3}
Suppose that $h \in H^1(\cK^*_n, E_{\pi^{n}})$. Then $h \in \ms_{\fp}(\cK^*_n,
E_{\pi^{n}})$ if and only if, for each $\vs^* \in E_{\pi^{*n}}$, the
following local conditions are satisfied:

(a) $\tilde{h}(\vs^*) \in \cK_{n,v}^{*\times p^n}$ for all $v \mid \fp$;

(b) $p^n \mid v_{\cK^*_n}(\tilde{h}(\vs^*))$ for all $v \nmid \fp^*$.

(Note that we impose no local conditions at places lying above
$\fp^*$.)
\end{corollary}

\begin{proof} This follows from Lemmas \ref{L:kummer1} and
\ref{L:kummer2}, using the definition of the isomorphism
\ref{E:mystiso} (see also \cite[Lemme 3.8]{PR1}).
\end{proof}

\begin{proposition} \label{P:kummerinj}
There are natural injections
\begin{align*}
&\rho: \Hom(T^*,(U^{*}_{\infty,\fp} \otimes
\bQ)/\ov{\cE}^*_\infty)^{\Gal(\cK_\infty^*/K)} \hookrightarrow
\ch{\ms}_{\fp}(K,T), \\
&\rho^*: \Hom(T,(U_{\infty,\fp^*} \otimes
\bQ)/\ov{\cE}_\infty)^{\Gal(\cK_\infty/K)} \hookrightarrow
\ch{\ms}_{\fp^*}(K,T^*)
\end{align*}
\end{proposition}

\begin{proof}
The proof of this result is essentially the same, \textit{mutatis
  mutandis}, as that of \cite[Proposition 2.4]{R}. The map $\rho$ is
  defined as follows.

For any $f \in \Hom(T^*,(U^{*}_{\infty,\fp} \otimes
\bQ)/\ov{\cE}^*_\infty)^{\Gal(\cK_\infty^*/K)}$ and any integer $n
\geq 1$, we define $f_n \in \Hom(E_{\pi^{*n}},
\cE^{*}_{n}/\cE_{n}^{*p^n})^{\Gal(\cK^*_\infty/K)}$ to be the image of $f$
under the following composition of maps:
\begin{align*}
\Hom(T^*,(U^{*}_{\infty,\fp} \otimes
\bQ)/\ov{\cE}^*_\infty)^{\Gal(\cK_\infty^*/K)} &\to 
\Hom(T^*,(U^{*}_{n,\fp} \otimes
\bQ)/\ov{\cE}^*_n)^{\Gal(\cK_\infty^*/K)} \\
&\to \Hom(E_{\pi^{*n}}, \cE^*_n/\cE_{n}^{*p^n})^{\Gal(\cK^*_\infty/K)},
\end{align*}
where the first arrow is the map induced by the natural projection
$U^{*}_{\infty,\fp} \to U^{*}_{n,\fp}$, and the second arrow is
induced by raising to the $p^n$-th power in $U^{*}_{n,\fp}$.

We define
\begin{equation} \label{E:rho}
\rho(f):= [(p-1)(\pi^{*})^{n} r_{n}^{-1}(f_n)] \in \varprojlim_n
H^1(K,E_{\pi^n}).
\end{equation} 
It follows from \cite[Lemma 3.16]{PR1} that $\rho(f)$ does indeed lie
in $\varprojlim_nH^1(K,E_{\pi^n})$. It is not hard to check from the
definition that $\rho$ is injective. It follows from \cite[Theorem
3.1, Proposition 3.2 and Corollary 3.3]{A1} that $r_{n}^{-1}(f_n) \in
\ms_{\fp}(K,E_{\pi^n})$ if and only if the restriction of
$r_{n}^{-1}(f_n)$ to $H^1(\fK_\infty,E_{\pi^n})$ is unramified outside
$\fp^*$. It may be shown via an argument very similar to that given in
\cite[Lemmas 2.1 and 2.3]{R} that this in fact the case.
\end{proof}

We shall use Proposition \ref{P:kummerinj} to produce canonical
elements in restricted Selmer groups by applying $\rho$ and $\rho^*$
to certain Galois-equivariant homomorphisms that are constructed using
norm-coherent sequences of elliptic units (see Section
\ref{S:elliptic} below).

\section{The $p$-adic height pairing on restricted Selmer
  groups} \label{S:height}

Our goal in this section is to describe a local decomposition of the
$p$-adic height pairing
\begin{equation} \label{E:lazy}
[\,,\,]_{K,{\fp^*}}: \ch{\ms}_{\fp}(K,T) \times
\ch{\ms}_{\fp^*}(K,T^*) \to O_{K,{\fp^*}}.
\end{equation}
in terms of Artin symbols; this is analogous to the local
decomposition of the standard $p$-adic height pairing on true Selmer
groups described in \cite[Lemme 3.19]{PR1}. This local decomposition
will be used in Section \ref{S:kummer} to determine certain values of
$[\,,\,]_{K,{\fp^*}}$ explicitly; these will in turn play a key r\^ole
in showing that the pairing \eqref{E:lazy} is non-degenerate when
$r=0$ in Section \ref{S:special}.

We begin by recalling the outlines of the main steps in the
construction of $[\,,\,]_{K,{\fp^*}}$. For more complete details, we
refer the reader to \cite[\S4]{A1}.

Let $\cY(\cK_\infty^*)$ denote the Galois group over $\cK_\infty^*$ of
the maximal abelian pro-$p$ extension of $\cK_\infty^*$ that is
unramified away from $\fp$ and totally split at all places of
$\cK_\infty^*$ lying above $\fp^*$. The first step in the construction
of the pairing $[\,,\,]_{K,{\fp^*}}$ is the construction of an
isomorphism
\begin{equation} \label{E:keyiso}
\Psi_K: \ch{\ms}_{\fp}(K,T) \xrightarrow{\sim} 
\Hom(T^*, \cY(\cK_\infty^*))^{\Gal(\cK_{\infty}^{*}/K)}.
\end{equation}
This isomorphism is constructed as follows. Write $J_n$ for the group
of finite ideles of $\cK_n^*$, and let $V_n$ denote the subgroup of
$J_n$ whose components are equal to $1$ at all places dividing $\fp$
and are units at all places not dividing $\fp^*$. Set 
$$
C_n:= J_n/(V_n \cdot \cK_{n}^{*\times}),\qquad \Omega_n:= \prod_{v
  \mid \fp} \mu_{p^n}(\cK_{n,v}^{*}),
$$
and write $C_n(p)$ for the $p$-primary subgroup of
$C_n$. Note that the order of $\Omega_n$ remains bounded as $n$
varies. We view $\Omega_n$ as being a subgroup of $C_n(p)$ via the
obvious embedding of $\Omega_n$ into $J_n$.

Using Kummer theory (cf. Corollary \ref{C:kummer3} above), one
constructs an exact sequence
\begin{equation*}
0 \to \Hom(E_{\pi^{*n}}, \Omega_n)^{\Gal(\cK_{n}^{*}/K)} \to
\Hom(E_{\pi^{*n}},C_n)^{\Gal(\cK_{n}^{*}/K)} \xrightarrow{\eta_n}
\ms_{\fp}(K,E_{\pi^n}) \to 0
\end{equation*}
(see \cite[Proposition 4.6]{A1}). Let $\eta_n'$ denote the map
obtained from $\eta_n$ via passage to the quotient by
$\Ker(\eta_n)$. It may be shown that passing to inverse limits over
the maps $\eta_{n}'^{-1}$ yields an isomorphism
$$
\Xi_K: \varprojlim {\ms}_{\fp}(K,E_{\pi^{n}}) = \ch{\ms}_{\fp}(K,T)
\xrightarrow{\sim} \Hom(T^*, \varprojlim C_n(p))^{\Gal(\cK^*_\infty/K)},
$$
where the inverse limit $\varprojlim C_n(p)$ is taken with respect
to the norm maps $\cK_{n}^{*\times} \to \cK_{n-1}^{*\times}$. One then
shows via class field theory (along with the fact that the weak
$p$-adic Leopoldt conjecture holds for $K$) that there is an
isomorphism
\begin{equation} \label{E:classiso}
\Hom(T^*, \varprojlim C_n(p))^{\Gal(\cK^*_\infty/K)} \simeq
\Hom(T^*,\cY(\cK_\infty^*))^{\Gal(\cK_\infty^*/K)},
\end{equation}
and composing this last isomorphism with with $\Xi_K$ yields the
desired isomorphism $\Psi_K$.

Next, by suitably interpreting restricted Selmer groups in terms of
certain Galois groups (see \cite[Theorem 3.1]{A1}), one shows that
there is a natural homomorphism
$$
\beta_K: \Hom(T^*, \cY(\cK^*_\infty))^{\Gal(\cK^*_\infty/K)} \to
\Hom_{O_{K,\fp^*}}(\ch{\ms}_{\fp^*}(K,T^*) , O_{K,\fp^*}).
$$
We thus obtain a map
$$
\beta_K \circ \Psi_K: \ch{\ms}_{\fp}(K,T) \to
\Hom_{O_{K,\fp^*}}(\ch{\ms}_{\fp^*}(K,T^*) , O_{K,\fp^*}),
$$
and this yields the $p$-adic height pairing pairing
$$
[\,,\,]_{K,\fp^*}: \ch{\ms}_\fp(K,T) \times \ch{\ms}_{\fp^*}(K,T^*)
\to O_{K,\fp^*}
$$ 
on restricted Selmer groups.

In order to describe the local decomposition of $[\,,\,]_{K,\fp^*}$,
we must introduce some further notation.

Suppose that
$$
y = [y_n] \in \ch{\ms}_{\fp}(K,T), \qquad   y^* = [y^*_n] \in
\ch{\ms}_{\fp^* }(K,T^*).
$$
For each positive integer $n$, we define $q_n$ to be the map
$$
q_n: \ch{\ms}_{\fp}(K,T) \xrightarrow{\Psi_K} 
\Hom(T^*, \cY(\cK^*_\infty))^{\Gal(\cK^*_\infty/K)} \to
\Hom(E_{\pi^{*n}},C_n)^{\Gal(\cK_{\infty}^{*}/K)},
$$
where the second arrow is the natural quotient map afforded by the
isomorphism \eqref{E:classiso}.

For each $\vs^* \in E_{\pi^{*n}}$, let $n(\vs^*)$ denote the exact
power of $\pi^*$ that kills $\vs^*$. Let $S_{n,\vs^*}(y_n)$ denote any
representative of $\eta_{n}^{-1}(y_n)(\vs^*)$ in $J_n$. For each
finite place $v$ of $K$, define $\{y,y^*\}_{n,v}^{(\vs^*)}$ to be the
unique element of $O_K/\pi^{*n(\vs^*)}O_K$ such that
$$
\{y,y^*\}_{n,v}^{(\vs^*)} \cdot \vs^* = y_n^*([S_{n,\vs^*}(y_n)_v,
K_{v}^{\ab}/\cK_{n,v}]),
$$
where $[S_{n,\vs^*}(y_n)_v, K_{v}^{\ab}/\cK_{n,v}] \in
\Gal(K_{v}^{\ab}/K_{n,v})$ is the obvious local Artin symbol.

\begin{proposition} \label{P:height}
(cf. \cite[Lemma 3.19]{PR1})

(a) For any $\vs^* \in E_{\pi^{*n}}$, we have
\begin{equation} \label{E:art}
[y,y^*]_{K,\fp^*} \cdot \vs^* = y_n^*([q_n(y)(\vs^*), K^{\ab}/\cK_n^*]),
\end{equation}
where $[q_n(y)(\vs^*), K^{\ab}/\cK_n^*] \in \Gal(K^{\ab}/\cK_n^*)$ is
the obvious global Artin symbol.
\smallskip

(b) We have
\begin{equation} \label{E:locdec}
[y,y^*]_{K,\fp^*} \equiv \sum_v \{y,y^*\}_{n,v}^{(\vs^*)}
\pmod{\pi^{*n(\vs^*)} O_{K,\fp^*}},
\end{equation}
where the sum is over all finite places $v$ of $\cK_n^*$.
\end{proposition}

\begin{proof}
(a) This follows immediately from the following commutative diagram:
\begin{equation*}
\begin{CD}
\ch{\ms}_{\fp}(K,T) @>{\Psi_K}>> \Hom(T^*,
\cY(\cK^*_\infty))^{\Gal(\cK^*_\infty/K)} @>{\beta_K}>>
\Hom_{O_{K,\fp^*}}(\ch{\ms}_{\fp^*}(K,T^*) , O_{K,\fp^*}) \\
@VVV    @VVV   @VVV \\
\ms_{\fp}(K,E_{\pi^n}) @>{\eta_{n}'^{-1}}>>
\dfrac{\Hom(E_{\pi^{*n}},C_n)^{\Gal(\cK_{\infty}^{*}/K)}}{\Ker(\eta_n)} @>>>
\Hom(\ms_{\fp^*}(K,E_{\pi^{*n}}), O_K/\pi^{*n}O_K)
\end{CD}
\end{equation*}
In this diagram, the second arrow on the bottom row is induced by the map
\begin{equation*}
f \mapsto (c \mapsto \{\vs^* \mapsto c([f(\vs^*), K^{\ab}/\cK_{n}^{*}]) \}
), \qquad f \in \Hom(E_{\pi^{*n}},C_n)^{\Gal(\cK_{\infty}^{*}/K)}
\end{equation*}
and is well-defined because $[z, K^{\ab}/\cK_{n}^{*}] =0$ for all
$z \in \Omega_n$. Note also that here we have canonically identified
$O_K/\pi^{*n}O_K$ with $\Hom(E_{\pi^{*n}}, E_{\pi^{*n}})$ via the map
$$
\beta \mapsto \{\vs^* \mapsto \beta \cdot \vs^*\}.
$$
\smallskip

(b) This follows from the local decomposition of the global Artin
symbol afforded via class field theory, viz. if $\alpha \in J_n$, then
$$
[\alpha, K^{\ab}/\cK_n^*] = \prod_v [\alpha_v, K_{v}^{\ab}/ \cK_{n,v}^{*}],
$$
where the product is over all finite places $v$ of $\cK_n^*$.
\end{proof}

%%%%%%%%%%%%%%%%%%%%%%%%%%%%%%%%%%%%%%%%%%%%%%%%%%%%%%%%%%%%%%%%%%%%%%%%%
%%%%%%%%%%%%%%%%%%%%%%%%%%%%%%%%%%%%%%%%%%%%%%%%%%%%%%%%%%%%%%%%%%%%%%%%%

\section{Formal Groups} \label{S:formal}

The purpose of this section is to recall a number of facts concerning
formal groups, and to establish certain conventions that we shall use
in Section \ref{S:kummer}.

We fix a minimal Weierstrass model of $E$ over $O_{K,\fp}$, and we
write $\hat{E}$ for its associated formal group. Let $\hat{\bG}_m$
denote the formal group over $O_{K,\fp}$ associated to the
multiplicative group $\bG_m$. If $x$ is a point on $\bG_m$ or on $E$,
then we write $\hat{x}$ for the corresponding value of the parameter
on $\hat{\bG}_m$ or on $\hat{E}$. We denote the formal group logarithm
associated to $\hat{E}$ by $\lambda_{\wh{E}}(Z) = Z + (\text{higher
  order terms}) \in O_{K,\fp}[[Z]]$, and we write $\log_{E,\fp}$ for
the corresponding $\fp$-adic logarithm associated to $E$. We denote
the $\fp$-adic logarithm associated to $\bG_m$ by $\log_{\fp}$.

Recall that $\cO$ denotes the completion of the ring of integers of
$\cK_{\infty,\fp}^{*}$. Since $\hat{E}$ is a height one Lubin-Tate
formal group, we may fix an isomorphism
\begin{equation*}
\eta: \hat{\bG}_m \to \hat{E}, \qquad \eta \in \cO[[Z]].
\end{equation*}
As explained in \cite[\S6]{R}, this choice of isomorphism then yields:

(a) A generator $w^* = [w_n^*]$ of $T^*$ such that for every $n \geq
1$ and $\vs \in E_{\pi^n}$, we have
\begin{equation} \label{E:torgen}
\eta(\wh{e_n(\pi^{*-n} \vs, w_n^*)}) = \hat{\vs}.
\end{equation}

(b) A $\fp$-adic period $\Omega_{\fp}:= \eta'(0) \in \cO^{\times}$
which is such that
\begin{equation*}
\Omega_{\fp}^{\sigma} = \psi^*(\sigma^{-1}) \Omega_{\fp}
\end{equation*}
for every $\sigma \in \Gal(\ov{K}/K)$.

We fix a generator $w = [w_n]$ of $T$, and for each $n \geq 0$, we set
\begin{equation*}
\zeta_n:= e_n(\pi^{*-n} w_n, w_n^*);
\end{equation*}
so from \eqref{E:torgen} above, we have that
\begin{equation*}
\eta(\hat{\zeta}_n) = \hat{w}_n.
\end{equation*}
We note that for each integer $n \geq 1$, our choice of $\eta$ induces
an isomorphism
\begin{equation*}
\chi_n: \mu_{p^n} \xrightarrow{\sim}  E_{\pi^n};\qquad \zeta_n \mapsto w_n
\end{equation*}
which is $\Gal(\ov{K}/\cK_n^*)$-equivariant; this in turn induces an
isomorphism (which we denote by the same symbol)
\begin{equation} \label{E:chin}
\chi_n: H^1(\cN_{n,\fp}, \mu_{p^n}) \xrightarrow{\sim} 
H^1(\cN_{n,\fp}, E_{\pi^n}).
\end{equation}

\begin{proposition} \label{P:compare}
Recall that for each integer $n \geq 1$, $\fm_{n,\fp}$ denotes the
maximal ideal in the completion of the ring of integers of
$\cN_{n,\fp}$. 

(a) With notation as above, the following diagram commutes:
\begin{equation*}
\begin{CD}
H^1(\cN_{n,\fp}, \mu_{p^n})  @>{\chi_n}>{\sim}> H^1(\cN_{n,\fp},
E_{\pi^n}) \\
@AAA          @AAA \\
\dfrac{\hat{\bG}_m(\fm_{n,\fp})}{\hat{\bG}_m(\fm_{n,\fp})^{p^n}}
\simeq \dfrac{O_{\cN_{n,\fp}}^{\times}}{ O_{\cN_{n,\fp}}^{\times p^n}}
@>{\eta}>{\sim}> \dfrac{\hat{E}(\fm_{n,\fp})}{\fp^n
  \hat{E}(\fm_{n,\fp})}.
\end{CD}
\end{equation*}
(Here the vertical arrows denote the natural maps afforded by Kummer
theory on $\hat{\bG}_m$ and $\hat{E}$.)

(b) If $\hat{x} \in \hat{\bG}_m(\fm_{n,\fp})$, then
\begin{equation*}
\log_{E,\fp}(\eta(\hat{x})) \equiv \Omega_{\fp} \cdot
\log_{\fp}(\hat{x}) \pmod{\fm_{n,\fp}^{p^n}}
\end{equation*}
on $\hat{\bG}_m(\fm_{n,\fp})/ \hat{\bG}_{m}(\fm_{n,\fp})^{p^n}$.
\end{proposition}

\begin{proof} (a) This follows directly from the definitions of $\eta$ and
  $\chi_n$. 

(b) See the proof of \cite[Corollary 9.2]{R}. 
\end{proof}

%%%%%%%%%%%%%%%%%%%%%%%%%%%%%%%%%%%%%%%%%%%%%%%%%%%%%%%%%%%%%%%%%%%%%%%%%%%
%%%%%%%%%%%%%%%%%%%%%%%%%%%%%%%%%%%%%%%%%%%%%%%%%%%%%%%%%%%%%%%%%%%%%%%%%%%

\section{Kummer pairings and $p$-adic heights} \label{S:kummer}

In this section we shall compute the value of the $p$-adic height
pairing $[\,,\,]_{K,\fp}$ on elements of restricted Selmer groups that
are constructed via Proposition \ref{P:kummerinj} (see Theorem
\ref{T:genht} below). This is accomplished by using Proposition
\ref{P:height} to express these values in terms of certain Kummer
pairings, and then applying Wiles's explicit reciprocity law.

\begin{definition} \label{D:kummer}
For each integer $n \geq 1$, we define a pairing
\begin{equation} \label{E:pairing1}
(\,,\,)_{\fp, E_{\pi^{*n}}}: U_{n,\fp}^{*} \times
    H^1(K_\fp,E_{\pi^{*n}}) \to E_{\pi^{*n}}
\end{equation}
by
\begin{equation*}
(u_n,c^*_n)_{\fp, E_{\pi^{*n}}} = c^*_n([u_n,
    K_{\fp}^{\ab}/\cK_{n,\fp}^{*}]) 
=: \alpha_{\fp,E_{\pi^{*n}}}(u_n,c^*_n) \cdot w_n^*,
\end{equation*}
with $\alpha_{\fp,E_{\pi^{*n}}}(u_n,c^*_n) \in O_K/\fp^{*n}O_K \simeq
\bZ_p/p^n \bZ_p$.

The pairings \eqref{E:pairing1} give a pairing
\begin{equation} \label{E:pairing2}
(\,,\,)_{\fp, T^*}: U_{\infty,\fp}^{*} \times
    H^1(K_\fp,T^*) \to T^*;\quad (u,c^*) \mapsto
    \alpha_{\fp,T^*}(u,c^*) \cdot w^*
\end{equation}
that is defined as follows. Suppose that $u =[u_n] \in
U_{\infty,\fp}^{*}$ and 
$$ 
c^* = [c^*_n] \in \varprojlim
H^1(K_\fp,E_{\pi^{*n}}) \simeq H^1(K_\fp, T^*).
$$ 
Then we set
\begin{equation*}
(u,c^*)_{\fp, T^*} = [(u_n,c^*_n)_{\fp, E_{\pi^{*n}}}] = [\alpha_{\fp,
E_{\pi^{*n}}} (u_n,c_n^*)] \cdot w^* =: \alpha_{\fp,T^*}(u,c^*) \cdot w^*. 
\end{equation*}
\qed
\end{definition}

The pairing \eqref{E:pairing2} is related to the $p$-adic height
pairing $[\, ,\,]_{K,\fp^*}$ in the following way. Recall from
Proposition \ref{P:kummerinj} above that there is a natural injection
$$
\rho: \Hom(T^*,(U^{*}_{\infty,\fp} \otimes
\bQ)/\ov{\cE}^*_\infty)^{\Gal(\cK_\infty^*/K)} \hookrightarrow
\ch{\ms}_{\fp}(K,T).
$$

\begin{proposition} \label{P:kummerht}
Suppose that 
\begin{equation*}
\xi \in
\Hom(T^*,U^{*}_{\infty,\fp}/\ov{\cE}^*_\infty)^{\Gal(\cK_\infty^*/K)}.
\end{equation*}
Let $\ov{\xi(w^*)}\in U_{\infty,\fp}^{*}$ denote any lift of
$\xi(w^*)$, and suppose that $y^* \in \ch{\ms}_{\fp^*}(K,T^*)$. Then
we have
$$
[\rho(\xi),y^*]_{K,\fp^*} \cdot w^* = (\ov{\xi(w^*)},
\loc_{\fp}(y^*))_{\fp, T^*},
$$
and so it follows that
$$ 
[\rho(\xi),y^*]_{K,\fp^*} =
[\alpha_{\fp,E_{\pi^{*n}}}((\ov{\xi(w^*)})_n,
\loc_{\fp}(y_n^*))] = \alpha_{\fp,T^*}(\ov{\xi(w^*)},\loc_{\fp}(y^*)).
$$
\end{proposition}

\begin{proof}
This follows directly from Proposition \ref{P:height} and Definition
\ref{D:kummer}.
\end{proof}

Our goal in this section is to evaluate $[\rho(\xi),y^*]_{K,\fp^*}$
(see Theorem \ref{T:genht} below). The difficulty in doing this arises
from the fact that that the pairing \eqref{E:pairing2} cannot be
directly evaluated using explicit reciprocity laws. In order to
overcome this obstacle, we shall first express the Kummer pairings
\eqref{E:pairing1} and \eqref{E:pairing2} in terms of certain Hilbert
symbols. This will then enable us to relate these pairings to two
other pairings (namely \eqref{E:nk1} and \eqref{E:nk2} below) that can
be evaluated using Wiles's explicit reciprocity law for formal groups.

\begin{definition} \label{D:hilbert}
We define the Hilbert pairing 
\begin{equation} \label{E:hilb1}
(\,,\,)_{\fp,\mu_{p^n}}:
  \frac{\cN_{n,\fp}^{\times}}{{\cN_{n,\fp}^{\times p^n}}} \times
  \frac{\cN_{n,\fp}^{\times}}{{\cN_{n,\fp}^{\times p^n}}} \to
  \mu_{p^n}
\end{equation}
by
\begin{equation*}
(a,b)_{\fp,\mu_{p^n}} = \frac{(b^{1/p^n})^{\sigma_a}}{b^{1/p^n}},
\end{equation*}
where $\sigma_a$ denotes the local Artin symbol
$[a,K_{\fp}^{\ab}/\cN_{n,\fp}]$ and $b^{1/p^n}$ is any $p^n$-th root of
$b$ in $K_{\fp}^{\ab}$.

We remark that it is a standard property of the Hilbert pairing (see
e.g. \cite[Chapter XIV, \S2]{S1}) that
\begin{equation} \label{E:skewhilb}
(a,b)_{\fp,\mu_{p^n}} = (b,a)^{-1}_{\fp,\mu_{p^n}}.
\end{equation}
\qed
\end{definition}

\begin{lemma} \label{L:kumhilb}
Suppose that $u \in U_{\infty,\fp}^{*}$ and $c^* \in
H^1(K_{\fp},T^*)$. Then
\begin{equation*}
(u_n, r_n^*(c_n^*)(\pi^{*-n}w_n))_{\fp,\mu_{p^n}} =
  \alpha_{\fp,E_{\pi^{*n}}}(u_n,c^*_n) 
\cdot \zeta_n.
\end{equation*}

%Suppose that $u_n \in U_{\fp,n}^{*}$ and $c_n \in
%H^1(K_{\fp},E_{\pi^{*n}})$. Then
%\begin{equation*}
%\alpha_{\fp, w_{n}^{*}} (u_n,c_n) = \alpha_{\fp,\mu_{p^n}}(u_n,\bc_n)
%\end{equation*}
%in $O_K/\fp^{*n}O_K \simeq \bZ /p^n\bZ$.
\end{lemma}

\begin{proof} It follows from the definition of $r_n^*$ (see
 Corollary \ref{C:riso}) that 
\begin{align*}
(u_n, r_n^*(c_n^*)(\pi^{*-n}w_n))_{\fp,\mu_{p^n}}&= e_n( \pi^{*-n}
w_n, c^*_n([u_n,K_{\fp}^{\ab}/\cN_{n,\fp}])) \\ 
&= e_n(\pi^{*-n} w_n, c^*_n([u_n,K_{\fp}^{\ab}/\cK_{n,\fp}^{*}]))\quad
\text{(since $u_n \in \cK_{n,\fp}^{*}$)} \\ 
&= e_n(\pi^{*-n} w_n, \alpha_{\fp,E_{\pi^{*n}}}(u_n,c^*_n) \cdot w_n^*) \\ 
&= \alpha_{\fp,E_{\pi^{*n}}}(u_n,c^*_n) \cdot e_n(w_n, \pi^{-n} w_n^*) \\
&= \alpha_{\fp,E_{\pi^{*n}}}(u_n,c^*_n) \cdot \zeta_n,
\end{align*}
as claimed.
\end{proof}

\begin{definition} \label{D:normkum}
For each integer $n \geq 1$, we define a pairing
\begin{equation} \label{E:nk1}
(\,,\,)_{\fp,E_{\pi^n}}: \cN_{n,\fp}^{\times} \times
  H^1(\cN_{n,\fp},E_{\pi^n}) \to E_{\pi^n}
\end{equation}
by
\begin{equation*}
(\nu, y)_{\fp,E_{\pi^n}} 
= y([\nu, K_{\fp}^{\ab}/\cN_{n,\fp}]). 
\end{equation*}

It is not hard to check (using Lemma \ref{L:resinj}) that the pairings
$(\,,\,)_{\fp,E_{\pi^n}}$ combine to yield a pairing
\begin{equation} \label{E:nk2}
(\,,\,)_{\fp,T}: U_{\infty,\fp} \times H^1(K_{\fp}, T) \to T;\quad
  (\beta,c) \mapsto \alpha_{\fp,T}(\beta,c) \cdot w
\end{equation}
that is defined as follows. Suppose that $\beta = [\beta_n] \in
U_{\infty,\fp}$, and that
\begin{equation*}
c = [c_n] \in \varprojlim H^1(K_{\fp}, E_{\pi^n}) \simeq H^1(K_{\fp},T).
\end{equation*}
Then we set
\begin{equation*}
(\beta,c)_{\fp,T} = [(\beta_n,c_n)_{\fp, E_{\pi^n}}] =
  [\alpha_{\fp,E_{\pi^n}}(\beta_n,c_n)] \cdot w =: \alpha_{\fp,T}(\beta,c)
  \cdot w.  \qed
\end{equation*}
\end{definition}

The advantage of the pairings \eqref{E:nk1} and \eqref{E:nk2} is that
they can be evaluated using explicit reciprocity laws. The following
result gives the relationship between the pairings \eqref{E:pairing1},
\eqref{E:hilb1} and \eqref{E:nk1}.

\begin{lemma} \label{L:htkum}
(a) Suppose that $c \in H^1(K_{\fp},T)$, and that $u \in
  U_{\infty,\fp}$. Then
\begin{equation*}
(u_n, r_n(c_n)(\pi^{-n} w_n^*))_{\fp,\mu_{p^n}} =
  \alpha_{\fp,E_{\pi^n}}(u_n,c_n) \cdot \zeta_n.
\end{equation*}

(b) Suppose that $u \in U_{\infty,\fp}^{*}$ and $c^* \in
H^1(K_{\fp},T^*)$. Then
\begin{equation*}
(u_n, \chi_n(r_n^*(c_n^*(\pi^{*-n} w_n))))_{\fp,E_{\pi^n}} =
-(r_n^*(c_n^*)(\pi^{*-n} w_n), \chi_n(u_n))_{\fp,E_{\pi^n}}=
\alpha_{\fp,E_{\pi^{*n}}}(u_n,c_n) \cdot w_n.
\end{equation*}
(Recall that the isomorphism $\chi_n$ is defined in \eqref{E:chin}.)
\end{lemma}

\begin{proof} 
Part (a) may be proved in exactly the same way as Lemma
\ref{L:kumhilb}, while Part (b) is an immediate consequence of 
the same lemma.
\end{proof}

In order to evaluate the pairings in Lemma \ref{L:htkum}(b) above
using explicit reciprocity laws (and thereby also compute the value
of the $[\rho(\xi),y^*]_{K,\fp^*}$ in Proposition \ref{P:kummerht} above), we
shall require the following result (cf. Corollary \ref{C:riso}).

\begin{lemma} \label{L:kumhom}
There are injective homomorphisms
\begin{align} \label{E:kumhom1}
&\kappa: H^1(K,T) \rightarrow \varprojlim H^1(\cK_n^*, \bZ_p(1))  \\ 
&\kappa^*: H^1(K,T^*) \rightarrow \varprojlim H^1(\cK_n, \bZ_p(1)),
\end{align}
where the inverse limits are taken with respect to the obvious
corestriction maps. A similar result holds if $K$ is replaced by
$K_{\fp}$ or $K_{\fp^*}$.
\end{lemma}

\begin{proof}
We shall just explain the construction of $\kappa^*$, since
that of $\kappa$ is very similar.

Recall that $w=[w_n]$ denotes a generator of $T$. Suppose that
$$
c^* = [c^*_n] \in \varprojlim H^1(K_\fp, E_{\pi^{*n}}) \simeq
H^1(K_\fp,T^*),
$$
and consider the composition of maps
\begin{equation} \label{E:redmap}
H^1(\cK_{n+1,\fp}, \mu_{p^{n+1}}) \to H^1(\cK_{n,\fp}, \mu_{p^{n+1}}) \to
 H^1(\cK_{n,\fp}, \mu_{p^{n}}),
\end{equation}
where the first arrow is given by corestriction and the second arrow
is induced by the natural map $\mu_{p^{n+1}} \to \mu_{p^{n}}$.  Recall
from Corollary \ref{C:riso} that there is an isomorphism
\begin{equation} \label{E:sriso}
r_n^*: H^1(K,E_{\pi^{*n}}) \xrightarrow{\sim} \Hom(E_{\pi^n},
\cK_{n}^{\times}/\cK_{n}^{\times p^n})^{\Gal(\cK_n/K)}.
\end{equation}
It is not hard to check that \eqref{E:redmap} maps
$\{r_{n+1}^{*}(c^{*}_{n+1})\}(\pi^{*-(n+1)} w_{n+1})$ to
$\{r_{n}^{*}(c^{*}_{n})\}(\pi^{*-n} w_n)$. 
We may therefore define
$$
\tilde{c^*}(w):= [\{r_n^*(c^*_n)\}(\pi^{*-n} w_n)] \in \varprojlim
H^1(\cK_{n,\fp}, \mu_{p^n}),
$$
where the second inverse limit is taken with respect to the maps
\eqref{E:redmap}. It is a standard fact (see e.g. \cite[Appendix B,
Section B.3]{R3}) that
$$
\varprojlim H^1(\cK_{n,\fp}, \mu_{p^n}) \simeq \varprojlim
H^1(\cK_{n,\fp}, \bZ_p(1)),
$$
where the right-hand inverse limit is taken with respect to the
obvious corestriction maps. We may therefore view $\tilde{c^*}(w)$ as
being an element of $\varprojlim H^1(\cK_{n,\fp}, \bZ_p(1))$, and we
write $\kappa^*(c^*) = [\kappa^*(c^*)_n] \in \varprojlim
H^1(\cK_{n,\fp}, \bZ_p(1))$ for this element. We remark that it
follows from the construction of $\kappa^*$ that, for each integer $n
\geq 1$, we have:
\begin{equation} \label{E:kappacong}
\kappa^*(c^*)_n \equiv \{r_n^*(c^*_n)\}(\pi^{*-n} w_n)
\pmod{\cK_{n,\fp}^{\times p^n}}
\end{equation}
in $H^1(\cK_{n,\fp}, \mu_{p^n}) \simeq
\cK_{n,\fp}^{\times}/\cK_{n,\fp}^{\times p^n}$. This implies that
$\kappa^*$ is injective.
\end{proof}

Suppose now that $u \in U_{\infty, \fp}^{*}$ and that $c^* \in
H^1(K_{\fp}, T^*)$ with $\kappa^*(c^*) \in U_{\infty,\fp} \subseteq
\varprojlim H^1(\cK_{n,\fp}, \bZ_p(1))$. Lemma \ref{L:htkum}(b)
implies that, for each integer $n \geq 1$,
$\alpha_{\fp,E_{\pi^{*n}}}(u_n,c^*_n)$ may be determined by
calculating the value of $(\kappa^*(c^*)_n,
\chi_n(u_n))_{\fp,E_{\fp^n}}$; this in turn may be done by using
Wiles's explicit reciprocity law applied to the formal group $\hat{E}$
associated to $E$, as we shall now describe.

For any $\beta=[\beta_n] \in U_{\infty,\fp}$, let $g_{\beta,w}(Z) \in
O_{K,\fp}[[Z]]$ denote the Coleman power series of $\beta$; so, for
every integer $n >0$, we have
\begin{equation} \label{E:col1}
g_{\beta,w}(\hat{w}_n) = \beta_n.
\end{equation}
Set
\begin{equation} \label{E:col2}
\delta g_{\beta,w}(Z):=
  \frac{g'_{\beta,w}(Z)}{g_{\beta,w}(Z) \cdot 
  \lambda'_{\wh{E}}(Z)},
\end{equation}
and write
\begin{equation} \label{E:col3}
\delta_w(\beta) = \delta g_{\beta,w}(0):=
  \frac{g'_{\beta,w}(0)}{g_{\beta,w}(0) \cdot 
  \lambda'_{\wh{E}}(0)} = \frac{g'_{\beta,w}(0)}{g_{\beta,w}(0)},
\end{equation}
(where for the last equality we have used the fact that $\lambda_{\wh{E}}(Z)
= Z + (\text{higher order terms})$).

\begin{remark} \label{R:delrem}
Suppose that $c^* \in H^1(K_{\fp}, T^*)$ with $\kappa^*(c^*) \in
U_{\infty,\fp} \subseteq \varprojlim H^1(\cK_{n,\fp}, \bZ_p(1))$. In
Proposition \ref{P:kappacup}, we shall express
$\delta_w(\kappa^*(c^*))$ in terms of $\exp_{\fp}^{*}(c^*)$, and in
Proposition \ref{P:value}(a), we shall show that, for suitably chosen
$\beta$, we have that $\delta_w(\beta)$ is equal to a non-zero
multiple of $\cL_{\fp}(\psi)$.  This will enable us to deduce that the
canonical class $s_{\fp^*} \in \ch{\ms}_{\fp^*}(K,T^*)$ is of infinite
order when $\cL_{\fp}(\psi) \neq 0$ (see Theorem \ref{T:infinite}).
\qed
\end{remark}

\begin{proposition} \label{P:wiles} 
Suppose that $u \in U_{\infty, \fp}^{*}$ and that $c^* \in
H^1(K_{\fp}, T^*)$ with $\kappa^*(c^*) \in U_{\infty,\fp} \subseteq
\varprojlim H^1(\cK_{n,\fp}, \bZ_p(1))$. Then, with notation as
above, we have

\begin{align} \label{E:i}
\alpha_{\fp,E_{\pi^{*n}}}(u_n, c^*_n) &\equiv
\left(\frac{\psi^*(\fp)}{p} -1 \right) \cdot \delta_w(\kappa^*(c^*))
\cdot \Omega_{\fp} \notag \\ 
&\times \sum_{\sigma \in \Gal(\cK_{n,\fp}^{*}/K_{\fp})} \psi^*(\sigma^{-1})
\log_{\fp}(u_{n}^{\sigma}) \pmod{\fp^n},
\end{align}
where $\Omega_{\fp} \in \cO^{\times}$ is the $\fp$-adic period described in
Section \ref{S:formal}.

Hence we have $(u,c^*)_{\fp,T^*} = \alpha_{\fp,T^*}(u,c^*)
  \cdot w^*$, where 
\begin{align} \label{E:ii}
\alpha_{\fp,T^*}(u,c^*) &=
\left(\frac{\psi^*(\fp)}{p} -1 \right) \cdot \delta_w(\kappa^*(c^*)) 
\cdot \Omega_{\fp} \notag \\
&\times \lim_{n \to \infty} \left\{ \sum_{\sigma \in
  \Gal(\cK_{n,\fp}^{*}/K_{\fp}) } \psi^*(\sigma^{-1})
\log_{\fp}( u_{n}^{\sigma} ) \right\}.
\end{align}

\end{proposition}

\begin{proof} Applying Wiles's explicit reciprocity law (see
\cite[Chapter I, Theorem 4.2]{dS}) to evaluate $(\kappa^*(c^*)_n,
\chi_n(u_n))_{\fp,E_{\pi^n}}$, we see (using Lemma \ref{L:htkum})
that the value of $\alpha_{\fp,E_{\pi^{*n}}}(u_n, c^*_n)$ is given by:
\begin{equation} \label{E:newfirst}
\alpha_{\fp,E_{\pi^{*n}}}(u_n, c^*_n) \equiv
-\Tr_{\cK_{n,\fp}^{*}/K_{\fp}} \left\{ \pi^{-n}  
\Tr_{\fK_{n,\fp}/\cK_{n,\fp}^{*}} \left( \log_{E,\fp}(\wh{\chi_n(u_n)}))
\cdot \delta g_{\kappa^*(c^*),w}(\hat{w}_n) \right) \right\} \pmod{\fp^n}. 
\end{equation}

Since $u_n \in U_{n,\fp}^{*}$ and $\chi_n$ is
$\Gal(\ov{K}/\cK_n^*)$-equivariant, it follows that
$\log_{E,\fp}(\wh{\chi_n(u_n)}) \in \cK_{n,\fp}^{*}$. Also, we have
that $\pi^{-n} \Tr_{\fK_{n,\fp}/\cK_{n,\fp}^{*}} (\delta
g_{\kappa^*(c^*),w}(\hat{w}_n)) \in K_{\fp}$ because $\hat{w}_n \in
\cK_{n,\fp}$. Hence \eqref{E:newfirst} implies that

\begin{equation} \label{E:first}
\alpha_{\fp,E_{\pi^{*n}}}(u_n, c^*_n)
\equiv -\left\{ \pi^{-n} \Tr_{\fK_{n,\fp}/\cK_{n,\fp}^{*}}( \delta
g_{\kappa^*(c^*)}(\hat{w}_n) ) \right\} \cdot \Tr_{\cK_{n,\fp}^{*}/K_{\fp}}
(\log_{E,\fp}(\wh{\chi_n(u_n)})) \pmod{\fp^n}.
\end{equation}

From Proposition \ref{P:compare}(b), we see that
\begin{equation*}
\log_{E,\fp}(\wh{\chi_n(u_n)}) \equiv \Omega_{\fp}\cdot
\log_{\fp}(u_n) \pmod{\fp^n}, 
\end{equation*}
and this implies that
\begin{equation} \label{E:second}
\Tr_{\cK_{n,\fp}^{*}/K_{\fp}}(\log_{E,\fp}(\wh{\chi_n(u_n)}) \equiv
\Omega_{\fp} \sum_{\sigma \in \Gal(\cK_{n,\fp}^{*}/K_{\fp})} \psi^{*}
(\sigma^{-1} \log_{\fp}(u_{n}^{\sigma})) \pmod{\fp^{n+1}}.
\end{equation}

We also observe that \cite[Chapter II, Proposition 4.5(iii)]{dS}
implies that
\begin{equation} \label{E:third}
\pi^{-n} \Tr_{\fK_{n,\fp}/\cK_{n,\fp}^{*}}(\delta
g_{\kappa^*(c^*),w}(\hat{w}_n))  
= \left(1- \frac{\psi^*(\fp)}{p} \right) \cdot \delta_w(\kappa^*(c^*)),
\end{equation}
which is independent of $n$. The congruence \eqref{E:i} now follows from
\eqref{E:first}, \eqref{E:second} and \eqref{E:third}. The equality
\eqref{E:ii} is a direct consequence of \eqref{E:i}.
\end{proof}

Suppose now that
$$
\xi \in \Hom(T^*,
U_{\infty,\fp}^{*}/\ov{\cE_{\infty}^{*}})^{\Gal(\cK_{\infty}^{*}/K)},  
\quad
\xi^* \in \Hom(T, U_{\infty,\fp^*}/\ov{\cE_{\infty}})^{\Gal(\cK_{\infty}/K)},
$$
and set
$$
c:= \rho(\xi) \in \ch{\ms}_{\fp}(K,T), \quad c^*:= \rho^*(\xi^*) \in
\ch{\ms}_{\fp^*}(K,T^*).
$$ 
Then the definition of $\rho^*$ implies that $c^* =[c_n^*] =
[(p-1)\pi^nr_{n}^{*-1}(\xi_n^*)] \in \varprojlim_n
H^1(K,E_{\pi^n})$. For any $\vs \in E_{\pi^n}$, a routine computation shows that
$$
r_n^*(c_n^*)(\vs) = \xi_n^*((p-1) \pi^{*n} \vs),
$$
and so setting $\vs = \pi^{*-n} w_n$, we see that
$$
r_n^*(c_n^*)(\pi^{*-n}w_n) = \xi_n^*(w_n)^{p-1}.
$$
Hence it follows from \eqref{E:kappacong} that
\begin{equation} \label{E:kappacong2}
\kappa^*(c^*)_n \equiv \xi_n^*(w_n)^{p-1} \pmod{\cK_{n,\fp}^{\times p^n}}
\end{equation}
in $H^1(\cK_{n,\fp}, \mu_{p^n}) \simeq
\cK_{n,\fp}^{\times}/\cK_{n,\fp}^{\times p^n}$. 

Recall from Section \ref{S:intro} that $\vt$ is a certain generator of
the ideal $\cI$ of $\Lambda(\cK_{\infty})$. We now observe that as
$\xi^*$ is $\Gal(\cK_{\infty}/K)$-equivariant, it follows that
$\xi^*(w)^{\vt} \in \ov{\cE}_{\infty}/ \cI \ov{\cE}_{\infty}$. The
following result describes the relationship between $\kappa^*(c^*)$
and $\xi^*(w)^{\vt}$.

\begin{proposition} \label{P:pain}
With the above notation, we have 
\begin{equation*}
\delta_w(\kappa^*(c^*)) = \delta_w(\xi^*(w)^{\vt}),
\end{equation*}
where $\delta_w(\xi^*(w)^{\vt})$ denotes $\delta_w$ applied to any
lift in $\ov{\cE}_{\infty}$ of $\xi^*(w)^{\vt} \in \ov{\cE}_{\infty}/
\cI \ov{\cE}_{\infty}$.  
\end{proposition}

\begin{proof}
Suppose that $\xi^*(w) = \alpha = [\alpha_n]$, with $\alpha_n \in
U_{n,\fp^*}/\ov{\cE}_n$, and set $\xi^*(w)^{\vt} = \beta = [\beta_n]$, with
$\beta_n \in \ov{\cE}_n/\cI \ov{\cE}_n$. It follows from the
definition of $\xi^*_n$ (see Proposition \ref{P:kummerinj}) that 
\begin{equation*}
\xi_n^*(w_n) = \alpha^{p^n} \in \ov{\cE}_n/\ov{\cE}_{n}^{p^n} =
\cE_n/\cE_{n}^{p^n}. 
\end{equation*}
For each $n \geq 1$, let $\vt_n$ denote the projection of $\vt$ to
$\bZ_p[\Gal(\cK_n/K)]$. It follows from \cite[Lemma 6.3]{R} that
\begin{equation} \label{E:theta}
\vt_n \sum_{\sigma \in \Gal(\cK_n/K)} \psi^{-1}(\ov{\sigma}) \sigma
\equiv -(p-1)p^n \pmod{p^n \cI \bZ_p[\Gal(\cK_n/K)]},
\end{equation}
where $\ov{\sigma}$ denotes any lift of $\sigma \in \Gal(\cK_n/K)$ to
$\Gal(\cK_{\infty}/K)$. 

Hence we have (using additive notation for Galois action):
\begin{align*}
\vt \xi_n^*(w_n) &= p^n \beta_n \\
&\equiv -(p-1)^{-1} \vt \left( \sum_{\sigma \in \Gal(\cK_n/K)}
\psi^{-1}(\ov{\sigma}) \sigma \right) \beta_n \pmod{p^n \vt
  \ov{\cE}_n}.
\end{align*}
Since $\ov{\cE}_n$ has no $\vt$-torsion (see \cite[Chapter III,
  Proposition 1.3]{dS}), we can divide this relation
by $\vt$ and apply \eqref{E:kappacong2} to obtain
\begin{align*}
(p-1) \xi_n^*(w_n) &\equiv \left( \sum_{\sigma \in \Gal(\cK_n/K)}
\psi^{-1}(\ov{\sigma}) \sigma \right) \beta_n \pmod{p^n \ov{\cE}_n} \\
&\equiv \kappa^*(c^*)_n \pmod{p^n \ov{\cE}_n}.
\end{align*}
It now follows from \cite[Chapter II, Lemma 4.8]{dS} that
\begin{equation*}
\delta_w(\kappa^*(c^*)) = \delta_w(\beta),
\end{equation*}
and this implies the desired result.
\end{proof}

Let $\ov{\xi(w^*)} \in U_{\infty,\fp}^{*}$ be any lift of $\xi(w^*)$.

\begin{theorem} \label{T:genht}
(a) For any $y^* \in \ch{\ms}_{\fp^*}(K,T^*)$, we have
\begin{align*}
[\rho(\xi), y^*]_{K,\fp^*} &= 
(p-1) \cdot \left(\frac{\psi^*(\fp)}{p} -1 \right) \cdot \delta_w(\kappa^*(y^*)) 
\cdot \Omega_{\fp} \notag \\
&\times \lim_{n \to \infty} \left\{ \sum_{\sigma \in
  \Gal(\cK_{n,\fp}^{*}/K_{\fp})} \psi^*(\sigma^{-1})
\log_{\fp}( (\ov{\xi(w^*)})_{n}^{\sigma} )\right\}.
\end{align*}

(b) We have

\begin{align*}
[\rho(\xi),\rho^*(\xi^*)]_{K,\fp^*} &= 
(p-1) \cdot \left(\frac{\psi^*(\fp)}{p} -1 \right) \cdot \delta_w(\xi^*(w)^{\vt}) 
\cdot \Omega_{\fp} \notag \\
&\times \lim_{n \to \infty} \left\{ \sum_{\sigma \in
  \Gal(\cK_{n,\fp}^{*}/K_{\fp})} \psi^*(\sigma^{-1})
\log_{\fp}( (\ov{\xi(w^*)})_{n}^{\sigma} )\right\}.
\end{align*}
\end{theorem}

\begin{proof} 
This result follows directly from Propositions \ref{P:kummerht},
\ref{P:wiles} and \ref{P:pain}.
\end{proof}

%%%%%%%%%%%%%%%%%%%%%%%%%%%%%%%%%%%%%%%%%%%%%%%%%%%%%%%%%%%%%%%%%%%%%%
%%%%%%%%%%%%%%%%%%%%%%%%%%%%%%%%%%%%%%%%%%%%%%%%%%%%%%%%%%%%%%%%%%%%%%%%

\section{The Dual Exponential Map} \label{S:dual}

Suppose that $c^* \in H^1(K_{\fp},T^*)$ with $\kappa^*(c^*) \in
U_{\infty,\fp} \subseteq \varprojlim H^1(\cK_{n,\fp}, \bZ_p(1))$, and
let
\begin{equation*}
\exp_{\fp}^{*}: H^1(K_{\fp}, T^*) \to \bQ_p
\end{equation*}
denote the Bloch-Kato dual exponential map (see e.g. \cite[Section
  5]{R4} for an account of the dual exponential map associated to an
elliptic curve). In this section we shall evaluate
$\exp_{\fp}^{*}(c^*)$ in terms of the $\delta_w(\kappa^*(c^*))$
(cf. Remark \ref{R:delrem} above and Proposition \ref{P:kappacup}
below).

We begin by describing the relationship between the cup product
pairing
\begin{equation} \label{E:cup}
\cup: H^1(K_{\fp},T) \times H^1(K_{\fp},T^*) \to \bZ_p
\end{equation}
and the Kummer pairing $(\,,\,)_{\fp,T}$ of the previous
section. This result is probably well known, but we are not aware of a
written reference, and so we include a proof here.

\begin{proposition} \label{P:kumcup}
Suppose that $c=[c_n] \in H^1(K_{\fp},T)$ and $c^*=[c_n^*] \in
H^1(K_{\fp},T^*)$. Then
\begin{equation*}
(\kappa^*(c^*), c)_{\fp,T} = -(c \cup c^*) \cdot w,
\end{equation*}
i.e.
\begin{equation*}
\alpha_{\fp,T}(\kappa^*(c^*),c) =
      [\alpha_{\fp,E_{\pi^{n}}}(\kappa^*(c^*)_n ,c_n)] = -(c \cup c^*). 
\end{equation*}
\end{proposition}

\begin{proof}
Let
\begin{equation*}
\cup: H^1(K_{\fp},E_{\pi^n}) \times H^1(K_{\fp},E_{\pi^{*n}}) \to
\bZ/p^n\bZ
\end{equation*}
denote the cup product pairing `at level $n$' afforded by
\eqref{E:cup}. To prove the desired result, it suffices to show that,
for each integer $n \geq 1$, we have
\begin{equation*}
(r^*_n(c^*_n)(\pi^{-*n}w_n), c_n)_{\fp,E_{\pi^{n}}} = -(c_n \cup
c_n^*) \cdot w_n.
\end{equation*}

In order to do this, we first recall (see e.g. \cite[Chapter XIV]{S1})
that the Hilbert pairing $(\,,\,)_{\fp,\mu_{p^n}}$ may be identified
with a cup product pairing
\begin{equation*}
\cup: H^1(\cN_{n,\fp}, \mu_{p^n}) \times H^1(\cN_{n,\fp}, \mu_{p^n})
\to \bZ/p^n\bZ
\end{equation*}
so that
\begin{equation*}
(a,b)_{\fp,\mu_{p^n}} = (a \cup b) \cdot \zeta_n.
\end{equation*}
We also note that the $\Gal(\ov{K}/\cN_n)$-equivariant
isomorphisms
\begin{align*}
&E_{\pi^n} \xrightarrow{\sim} \mu_{p^n}; \quad w_n \mapsto e_n(w_n,
  \pi^{-n} w_n^*), \\
&E_{\pi^{*n}} \xrightarrow{\sim} \mu_{p^n}; \quad w^*_n \mapsto
e_n(\pi^{*-n}w_n, w_n^*) 
\end{align*}
induce isomorphisms of cohomology groups
\begin{align*}
&\kappa_n: H^1(\cN_{n,\fp},E_{\pi^n}) \xrightarrow{\sim}
  H^1(\cN_{n,\fp}, \mu_{p^n});\quad y_n \mapsto
  r_n(y_n)(\pi^{-n}w_n^*) ;\\
&\kappa^*_n: H^1(\cN_{n,\fp},E_{\pi^{*n}}) \xrightarrow{\sim}
  H^1(\cN_{n,\fp}, \mu_{p^n});\quad y^*_n \mapsto
  r^*_n(y^*_n)(\pi^{*-n}w_n),
\end{align*}
and via functoriality of cup product pairings, we have
\begin{equation*}
c_n \cup c_n^* = \kappa_n(c_n) \cup \kappa^*_n(c_n^*) =
-(\kappa_n^*(c_n^*) \cup \kappa_n(c_n)).
\end{equation*}
Hence, from Lemma \ref{L:htkum}(b), we see that
\begin{align*}
\alpha_{\fp, E_{\pi^{n}}}(r^*_n(c^*_n)(\pi^{-*n}w_n),c_n) 
&= \alpha_{\fp,\mu_{p^n}}(r^*_n(c^*_n)(\pi^{-*n}w_n),
r_n(c_n)(\pi^{-n} w^*_n)) \\
&= \kappa^*_n(c_n) \cup \kappa_n(c_n) \\
&= -(\kappa_n(c_n) \cup\kappa^*_n(c^*_n )) \\
&= -(c_n \cup c_n^*),
\end{align*}
as required.
\end{proof}
 
We can now state the main result of this section.

\begin{proposition} \label{P:kappacup}
Suppose that $c^* \in H^1(K_{\fp},T^*)$ with $\kappa^*(c^*) \in
U_{\infty,\fp} \subseteq \varprojlim H^1(\cK_{n,\fp}, \bZ_p(1))$. 
Then
\begin{equation*}
\exp_{\fp}^{*}(c^*) = \left( \frac{\psi^*(\fp)}{p} - 1
\right) \cdot \delta_w (\kappa^*(c^*)).
\end{equation*}
\end{proposition}

\begin{proof}
It follows from the definition of the dual exponential map (see
\cite[Section 5]{R4}) that
\begin{equation*}
x \cup c^* = \log_{E,\fp}(x) \cdot \exp_{\fp}^{*}(c^*)
\end{equation*}
for all $x \in H^1(K_{\fp},T)$. (Here we have extended the logarithm
map $\log_{E,\fp}$ from $E(K_{\fp}) \otimes \bZ_p$ to $H^1(K_{\fp},T)$
via linearity.)

From Wiles's explicit reciprocity law (see \cite[Chapter IV,
  \S2.5]{dS} for the version we require), we have
\begin{equation*}
(\kappa^*(c^*), x)_{\fp,T} = \left\{ \log_{E,\fp}(x) \cdot 
\left( 1 - \frac{\psi^*(\fp)}{p}\right) \cdot \delta_w (\kappa^*(c^*))
\right\} \cdot w.
\end{equation*}
The result now follows from Proposition \ref{P:kumcup}.
\end{proof}
%%%%%%%%%%%%%%%%%%%%%%%%%%%%%%%%%%%%%%%%%%%%%%%%%%%%%%%%%%%%%%%%%%%%%%%%%%%%
%%%%%%%%%%%%%%%%%%%%%%%%%%%%%%%%%%%%%%%%%%%%%%%%%%%%%%%%%%%%%%%%%%%%%%%%%%%%

\section{Elliptic units and canonical elements} \label{S:elliptic}

We shall now explain how elliptic units may be used, following the
methods described in \cite{R}, to construct canonical elements
$$
s_{\fp} \in \ch{\ms}_{\fp}(K,T),\quad s_{\fp^*} \in
\ch{\ms}_{\fp^*}(K,T^*)
$$
when $r=0$. These are the analogues in the present situation of the
elements $x_{\fp}^{(1)} \in \ch{\Sel}(K,T)$ and $x_{\fp^*}^{(1)} \in
\ch{\Sel}(K,T^*)$ constructed in \cite[\S4]{R} when $r=1$.

Let $\cC_\infty \subseteq \cE_\infty$ and $\cC^*_\infty \subseteq
\cE^*_\infty$ denote the norm-coherent systems of elliptic units
constructed in \cite[\S3]{R}, and write $\ov{\cC}_\infty$
and $\ov{\cC}^*_\infty$ for the closures of $\cC_\infty$ in
$\ov{\cE}_\infty$ and $\cC^*_\infty$ in $\ov{\cE}^*_\infty$
respectively. 
%Set
%$$
%\cI^*:= \Ker(\psi^*: \Lambda(\cK^*_\infty) \to \bZ_p), \quad
%\cI:= \Ker(\psi: \Lambda(\cK_\infty) \to \bZ_p),
%$$ 
%and let $\vt^*$ be the generator of $\cI^*$ fixed in \cite[\S6]{R}
%(so $\vt^* = \gamma \psi^*(\gamma^{-1}) -1$, where $\gamma$ is any
%topological generator of $\Gal(\cK_\infty^*/K)$ satisfying
%$\log_p(\psi^*(\gamma)) = p$). 
Recall that $\ff \subseteq O_K$ denotes
the conductor of the Grossencharacter associated to $E$, and fix $B \in
E_{\ff}/\Gal(\ov{K}/K)$ of exact order $\ff$ as described in
\cite[\S6]{R}. For each generator $w=[w_n]$ of $T$ and $w^*=[w_n^*]$
of $T^*$, let
\begin{align*}
&\theta_B(w)= [\theta_B(w_n)] \in \ov{\cC}_\infty \subseteq
U_{\infty,\fp} \otimes \bQ \\
&\theta_B(w^*)=[\theta_B(w_n^*)] \in \ov{\cC}^*_\infty \subseteq
U_{\infty,\fp}^{*} \otimes \bQ
\end{align*}
denote the norm-coherent sequences of elliptic units constructed in
\cite[\S3]{R}. In particular, we have that
\begin{equation} \label{E:ell1}
\theta_B(w)^{\sigma} = \theta_B(w^{\sigma}),\qquad
\theta_B(w^*)^{\sigma} = \theta_B(w^{*\sigma})
\end{equation}
for all $\sigma \in \Gal(\ov{K}/K)$, and
for each $n>0$, we have that
\begin{equation} \label{E:ell2}
\ov{\cC}_n= \bZ_p[\Gal(\cK_n/K)] \cdot \theta_B(w_n),\qquad
\ov{\cC}^*_n= \bZ_p[\Gal(\cK^*_n/K)] \cdot \theta_B(w^*_n)
\end{equation}
(see \cite[Propositions 3.1 and 3.2]{R1}).

The following result is due to Rubin (see \cite[Theorem 7.2]{R1}) and
it provides a crucial link between elliptic units and $p$-adic
L-functions. It is proved via a careful analysis of the construction
of the Katz two-variable $p$-adic L-function $\cL_{\fp}$ using
$p$-adic measures arising from elliptic units as described in
\cite[Chapter II, \S4]{dS}. We refer the reader to \cite[\S7]{R1} for
the details of the proof.

\begin{theorem} \label{T:keylink}
Suppose that $w$ is any generator of $T$. Then there are natural
Galois-equivariant injections
\begin{equation*}
\iota_w: U_{\infty,\fp} \to \Lambda(\cK_{\infty})_{\cO}, \qquad
\iota^*: U^{*}_{\infty,\fp} \to \Lambda(\cK^{*}_{\infty})_{\cO},
\end{equation*}
(where $\iota_w$ depends upon the choice of $w$, but $\iota^*$ does
not), satisfying the following properties. Suppose that $w^*$ is a
generator of $T^*$ fixed as in Section \ref{S:formal}. Then

(a) We have that
\begin{equation} \label{E:key1}
(1-\Fr_{\fp^*}) \iota_w(\theta_B(w)) = \cL_{\fp} |_{\cK_{\infty}},
\qquad
\iota^*(\theta_B(-\bN(\ff)^{-1} w^*)) = \cL_{\fp}
|_{\cK^{*}_{\infty}},
\end{equation}
where $\Fr_{\fp^*}$ denotes the Frobenius of $\fp^*$ in
$\Gal(\cK_{\infty}/K)$, and $\cL_{\fp} |_{\cK_{\infty}}$ and
$\cL_{\fp}|_{\cK^{*}_{\infty}}$ denote the images of $\cL_{\fp}$ under the
natural quotient maps $\Lambda(\fK_{\infty})_{\cO} \to
\Lambda(\cK_{\infty})_{\cO}$ and $\Lambda(\fK_{\infty})_{\cO} \to
\Lambda(\cK^{*}_{\infty})_{\cO}$ respectively.
\smallskip

(b) For every $\beta \in U_{\infty,\fp}$, we have
\begin{equation} \label{E:key2}
\iota_w(\beta)= \left( 1 - \frac{\psi^*(\fp)}{p} \right) \cdot
\Omega_{\fp} \cdot \delta_w(\beta).
\end{equation}
\smallskip

(c) For every $u=[u_n] \in U^{*}_{\infty,\fp}$, we have
\begin{equation} \label{E:key3}
\iota^*(u)(\psi^*) =
\left( 1 -
\frac{\psi^*(\fp)}{p} \right) \cdot
\lim_{n \to \infty} \left\{ \sum_{\tau \in \Gal(\cK_n^*/K)} \psi^{*-1}(\tau)
\log_{\fp} ( u_{n}^{\tau}) \right\}.
\end{equation}
\qed
\end{theorem}

Now suppose that $r=0$. Then $L_{\fp}(1) \neq 0$, and so we see from
\eqref{E:key1} that
$$
\ov{\cC}_\infty \subseteq \ov{\cE}_\infty \subset U_{\infty,\fp}
\otimes \bQ \quad \text{and} \quad \ov{\cC}_\infty \not \subseteq
\cI(U_{\infty,\fp} \otimes \bQ).
$$
In particular, we have that $\ov{\cC}_\infty \not \subseteq \cI
\ov{\cE}_\infty \subseteq U_{\infty,\fp} \otimes \bQ$. Similar remarks
imply that also $\ov{\cC}^*_\infty \not \subseteq \cI^*
\ov{\cE}^*_\infty \subseteq U^{*}_{\infty,\fp^*} \otimes
\bQ$. Applying Remark \ref{R:leo}, we deduce that
\begin{equation} \label{E:r1}
\ov{\cC}^*_\infty \not \subseteq \cI^* \ov{\cE}^*_\infty \subseteq
U_{\infty,\fp}^{*} \otimes \bQ.
\end{equation} 

Also, if $L_{\fp}(1) \neq 0$, then this implies that $L_{\fp}^{*}(1) =
0$, and so again from \eqref{E:key1}, it follows that we have
\begin{equation} \label{E:r2}
\ov{\cC}_\infty^* \subseteq \cI^* (U_{\infty,\fp}^{*} \otimes \bQ).
\end{equation}

Recall from Section \ref{S:intro} that $\vt^*$ is a certain generator of the
ideal $\cI^*$ of $\Lambda(\cK_{\infty}^{*})$.

\begin{proposition} \label{P:canhom}
There exists a unique homomorphism $\sigma_{\fp} \in
\Hom(T^*,(U_{\infty,\fp}^{*} \otimes
\bQ)/\ov{\cE}_\infty^*)^{\Gal(\cK_{\infty}^{*}/K)}$ such that
$$
\sigma_{\fp}(w^*)^{\vt^*} = \theta_B(w^*)
$$
in $\ov{\cE}_\infty^*/\cI^*\ov{\cE}_\infty^*$.
\end{proposition}

\begin{proof} The proof of this result is very similar to that of
  \cite[Theorem 4.2]{R}. We first observe that it follows from
  \eqref{E:ell1} and \eqref{E:ell2} that $\Gal(\cK_{\infty}^{*}/K)$
  acts on $T^*$ and $\ov{\cC}_{\infty}^{*}/ \cI^*
  \ov{\cC}_{\infty}^{*}$ via the same character $\psi^*$. Since
  $\Gal(\cK_{\infty}^{*}/K)$ also acts transitively on $T^* - pT^*$,
  it follows that the map $w^* \mapsto \theta_B(w^*)$ induces a
  well-defined homomorphism $\sigma_{\fp}^{(0)} \in \Hom(T^*,
  \ov{\cC}_{\infty}^{*}/ \cI^*
  \ov{\cC}_{\infty}^{*})^{\Gal(\cK_{\infty}^{*}/K)}$. 

Next, we note that it follows from \cite[Chapter III, Proposition
  1.3]{dS} that $U_{\infty,\fp}^{*}$ contains no $\vt^*$-torsion
elements, and so, if $\alpha \in \ov{\cC}_{\infty}^{*}$, then
$\vt^{-1} \alpha$ is a well-defined element of $U_{\infty,\fp}^{*}
\otimes \bQ$. It is not hard to show that the image of $\vt^{-1}
\alpha$ in $(U_{\infty,\fp}^{*} \otimes \bQ)/ \ov{\cE}^{*}_{\infty}$
depends only upon the image of $\alpha$ in
$\ov{\cC}^{*}_{\infty}/\cI^* \ov{\cC}^{*}_{\infty}$. The map
$\sigma_{\fp}$ is defined by $\sigma_{\fp}:= \vt^{-1} \sigma_{\fp}^{(0)}$.
\end{proof}

\begin{definition}
We set
\begin{equation*}
s_{\fp}:= \rho(\sigma_{\fp}) \in \ch{\ms}_{\fp}(K,T), \quad s_{\fp^*}:=
\rho^*(\sigma_{\fp^*})\in \ch{\ms}_{\fp^*}(K,T^*),
\end{equation*}
where of course the definition $\sigma_{\fp^*} \in
\Hom(T,(U_{\infty,\fp^*} \otimes
\bQ)/\ov{\cE}_{\infty})^{\Gal(\cK_{\infty}/K)}$ is the same, 
\textit{mutatis mutandis}, as that of $\sigma_{\fp}$. \qed
\end{definition}

\begin{remark} \label{R:inforder} We shall see from the proof of Theorem
  \ref{T:nondeg}(a) below that $s_{\fp^*}$ is of infinite order only
  if $\cL'_{\fp}(\psi^*) \neq 0$. \qed
\end{remark}

%%%%%%%%%%%%%%%%%%%%%%%%%%%%%%%%%%%%%%%%%%%%%%%%%%%%%%%%%%%%%%%%%%%%%%%%%%%%%
%%%%%%%%%%%%%%%%%%%%%%%%%%%%%%%%%%%%%%%%%%%%%%%%%%%%%%%%%%%%%%%%%%%%%%%%%%%%%

\section{Proof of Theorem \ref{T:A}} \label{S:special}

In this section we shall establish a number of properties of the
cohomology classes $s_\fp$ and $s_{\fp^*}$ constructed in Section
\ref{S:elliptic}, and thereby prove Theorem \ref{T:A} of the
Introduction. 

We begin by recalling some basic facts concerning derivatives of
elements in Iwasawa algebras (see \cite[Section 7]{R1} or
\cite[Section 2]{A1}, for example). Suppose that $F \in
\Lambda(\fK_{\infty})_{\cO}$, and consider the function
\begin{equation*}
s \mapsto F(\psi^* \langle \psi^* \rangle^s), \qquad s \in \bZ_p;
\end{equation*}
this is a $p$-adic analytic function on $\bZ_p$. For any integer $m
\geq 0$, define
\begin{equation*}
\bD^{*(m)}F(\psi^*) = \frac{1}{m!} \frac{d}{ds} F(\psi^* \langle
\psi^* \rangle^s) \Biggr |_{s=0};
\end{equation*}
then, by definition, we have
\begin{equation*}
\cL'_{\fp}(\psi^*) = \bD^{*(1)} \cL_{\fp}(\psi^*).
\end{equation*}
It is easy to check that
\begin{equation*}
\bD^{*(m)} \vt^{*m}(\psi^*) = p^m
\end{equation*}
for every integer $m \geq 1$.

The following result is \cite[Lemma 7.3]{R1}. It is proved via a
routine computation which we shall omit.

\begin{lemma} \label{L:dertheta}
With notation as above, we have
\begin{equation*}
\bD^{*(m)} (\vt^{*m}F)(\psi^*)  = p^m F(\psi^*)
\end{equation*}
for every integer $m \geq 1$. \qed
\end{lemma}

We now turn to the precise relationship between the canonical elements
$s_{\fp}$ and $s_{\fp^*}$, and certain special values of the $p$-adic
L-function $\cL_{\fp}$.

\begin{proposition} \label{P:value}
Let $\ov{\sigma_{\fp}(w^*)} \in U_{\infty,\fp}^{*}$ be any lift of
$\sigma_{\fp}(w^*)$. We have
\begin{equation} \label{E:vali}
\cL_{\fp}(\psi) = \left( 1 - \frac{1}{\psi(\fp^*)} \right) \cdot
\left( 1 - \frac{1}{\psi(\fp)} \right) \cdot 
\Omega_{\fp} \cdot \delta_w(\theta_B(w)).
\end{equation}
and
\begin{equation} \label{E:valii}
\cL'_{\fp}(\psi^*) = -p \cdot \bN(\ff)^{-1} \cdot \left( 1 -
\frac{\psi^*(\fp)}{p} \right) \cdot
\lim_{n \to \infty} \left\{ \sum_{\tau \in \Gal(\cK_n^*/K)} \psi^{*-1}(\tau)
\log_{\fp} ( \ov{\sigma_{\fp}(w^*)}_{n}^{\tau}) \right\}.
\end{equation}
\end{proposition}

\begin{proof} The equality \eqref{E:vali} follows directly from \eqref{E:key1}
and \eqref{E:key2}. To prove \eqref{E:valii}, we first observe that,
since there is no $\vt^*$-torsion in $U^{*}_{\infty,\fp}$ (see
\cite[Chapter III, Proposition 1.3]{dS}), there exists a unique
$\beta=[\beta_n] \in U^{*}_{\infty,\fp} \otimes_{\bZ} \bQ$ such that
\begin{equation*}
\beta^{\vt^*} = \theta_B(w^*).
\end{equation*}
Applying Theorem \ref{T:keylink}(a), we see that
\begin{align*}
- \vt^* \bN(\ff)^{-1} \iota^*(\beta) &=
\iota^*(\theta_B(\bN(\ff)^{-1}w^*) \\
&= \cL_{\fp}|_{\cK^{*}_{\infty}}.
\end{align*}
It therefore follows from Lemma \ref{L:dertheta} and Theorem
\ref{T:keylink}(c) that
\begin{align} \label{E:beta}
\cL'_{\fp}(\psi^*) &= -\bN(\ff)^{-1} \iota^*(\beta)(\psi^*) \notag \\ 
&=
-p \cdot \bN(\ff)^{-1} \cdot \left( 1 -
\frac{\psi^*(\fp)}{p} \right) \cdot
\lim_{n \to \infty} \left\{ \sum_{\tau \in \Gal(\cK_n^*/K)} \psi^{*-1}(\tau)
\log_{\fp} ( \beta_{n}^{\tau}) \right\}.
\end{align}

Now suppose that $\ov{\sigma_{\fp}(w^*)} \in U_{\infty,\fp}^{*}$ is
any lift of $\sigma_{\fp}(w^*)$. It follows from Proposition
\ref{P:canhom} that we have
\begin{equation} \label{E:betasigma}
\sigma_{\fp}(w^*)^{\vt^*} = \theta_B(w^*) = \beta^{\vt^*}
\end{equation}
in $\ov{\cE}^*_\infty/ \cI^* \ov{\cE}^*_\infty$. For each integer $n
\geq 1$, let $\vt_n^*$ denote the projection of $\vt$ to
$\bZ_p[\Gal(cK_n^*/K)]$. It is shown in \cite[Lemma 6.3]{R1} that
\begin{equation*}
\vt^*_n \sum_{\tau \in \Gal(\cK_n/K)} \psi^{-1}(\tau) \tau
\equiv -(p-1)p^n \pmod{p^n \cI \bZ_p[\Gal(\cK^*_n/K)]}.
\end{equation*}
We therefore deduce from \eqref{E:betasigma} that, for each $n \geq
1$, we have (writing Galois action additively)
\begin{equation*}
\vt^* \left( \sum_{\tau \in \Gal(\cK^*_n/K)} \psi^{*-1}(\tau) \tau
\right) \beta_n \equiv
\vt^* \left( \sum_{\tau \in \Gal(\cK^*_n/K)} \psi^{*-1}(\tau) \tau
\right) \ov{\sigma_{\fp}(w^*)}  \pmod{p^n \vt^* \ov{\cE}^*_n}.
\end{equation*}
Since $\ov{\cE}^*_n$ has no $\vt^*$-torsion (see \cite[Chapter III,
  Proposition 1.3]{R1}), we can divide both sides of this last
relation by $\vt^*$, and then take $\fp$-adic logarithms and let $n
\to \infty$ to deduce that
\begin{equation} \label{E:finaltouch}
\lim_{n \to \infty} \left\{ \sum_{\tau \in \Gal(\cK_n^*/K)} \psi^{*-1}(\tau)
\log_{\fp} ( \beta_{n}^{\tau}) \right\} 
=
\lim_{n \to \infty} \left\{ \sum_{\tau \in \Gal(\cK_n^*/K)} \psi^{*-1}(\tau)
\log_{\fp} ( \ov{\sigma_{\fp}(w^*)}_{n}^{\tau}) \right\}.
\end{equation}
The desired equality \eqref{E:valii} now follows from \eqref{E:beta}
and \eqref{E:finaltouch}.
\end{proof}

\begin{theorem} \label{T:infinite}
We have
\begin{align*}
\exp_{\fp}^{*}(s_{\fp^*}) &= \left( \frac{\psi^*(\fp)}{p} - 1
\right) \cdot \delta_w (\theta_B(w)) \\
&= \left( \frac{\psi^*(\fp)}{p} - 1 \right) \cdot \left( 1-
\frac{1}{\psi(\fp^*)} \right)^{-1} \cdot
\left( 1 - \frac{1}{\psi(\fp)} \right)^{-1} \cdot
\frac{\cL_{\fp}(\psi)}{\Omega_{\fp}}.  
\end{align*}
Hence, if $\cL_{\fp}(\psi) \neq 0$, then $s_{\fp^*}$ is of infinite
order.
\end{theorem}

\begin{proof} This follows from Propositions \ref{P:pain},
  \ref{P:kappacup} and \ref{P:value}.
\end{proof}

%\begin{theorem} \label{T:nonvanish}
%Suppose that $\cL_{\fp}(\psi) \neq 0$. Then $\cL'_{\fp}(\psi^*) \neq
%0$.
%\end{theorem}

%\begin{proof}
%Suppose that $\cL'_{\fp}(\psi^*) = 0$, and set $t = \ord_{s=1}
%\cL_{\fp}^{*}(s)$, so $t \geq 2$. Then it follows from \cite[Theorem
%  7.2(i)]{R} that
%\begin{equation*} 
%\ov{\cC}_{\infty}^{*} \subseteq \cI^{*t}(U_{\infty,\fp}^{*} \otimes
%\bQ).
%\end{equation*}
%Hence we have that $\sigma_{\fp^*}(w^*) \in \cI^{*(t-1)}(U_{\infty,\fp}^{*}
%\otimes \bQ) \cdot \ov{\cE}^{*}_{\infty}$, and this in turn implies
%that $s_{\fp^*} = 0$. This contradicts Theorem \ref{T:infinite}
%because by hypothesis $\cL_{\fp}(\psi) \neq 0$. It therefore follows
%that $\cL'_{\fp}(\psi^*) \neq 0$, as claimed.
%\end{proof} 

We can now prove the first part of Theorem \ref{T:A}.

\begin{theorem} \label{T:nondeg}
Suppose that $\cL_{\fp}(\psi) \neq 0$. 

(a) We have $\cL'_{\fp}(\psi^*) \neq 0$.
\smallskip

(b) The height pairing $[\,,\,]_{K,\fp}$ is non-degenerate, and we
have
\begin{equation*}
[s_{\fp},s_{\fp^*}]_{K,\fp^*} = 
(p-1) \cdot p^{-1} \cdot  \bN(\ff)^{-1} \cdot \left( 1 - \frac{1}{\psi(\fp^*)}
\right)^{-1} \cdot \left( 1 - \frac{1}{\psi(\fp)} \right)^{-1} \cdot
\cL_{\fp}(\psi) \cdot \cL'_{\fp}(\psi^*). 
\end{equation*}
\end{theorem}

\begin{proof}
(a) Suppose that $\cL'_{\fp}(\psi^*) = 0$, and set $t = \ord_{s=1}
\cL_{\fp}^{*}(s)$, so $t \geq 2$. Then it follows from \eqref{E:key1} that
\begin{equation*} 
\ov{\cC}_{\infty}^{*} \subseteq \cI^{*t}(U_{\infty,\fp}^{*} \otimes
\bQ).
\end{equation*}
Hence we have that $\sigma_{\fp^*}(w^*) \in \cI^{*(t-1)}(U_{\infty,\fp}^{*}
\otimes \bQ) \cdot \ov{\cE}^{*}_{\infty}$, and this in turn implies
that $s_{\fp^*} = 0$. This contradicts Theorem \ref{T:infinite}
because by hypothesis $\cL_{\fp}(\psi) \neq 0$. It therefore follows
that $\cL'_{\fp}(\psi^*) \neq 0$, as claimed.
\smallskip
 
(b) The equality follows directly from Theorem \ref{T:genht}(b),
Proposition \ref{P:value}, and Theorem \ref{T:infinite}. As
$[s_{\fp},s_{\fp^*}]_{K,\fp^*} \neq 0$, we deduce that
$[\,,\,]_{K,\fp^*}$ is non-degenerate because $\ch{\ms}_{\fp}(K,T)$
and $\ch{\ms}_{\fp^*}(K,T^*)$ are both free, rank one $\bZ_p$-modules
(see \cite[Lemma 3.6 and Proposition 6.7]{A1}).
\end{proof}

The second part of Theorem \ref{T:A} is a direct consequence of the
following result.

\begin{theorem} \label{T:quot}
(cf. \cite[Theorem 10.1]{R})

Suppose that $\cL_{\fp}(\psi) \neq 0$, and that $y \in
\ch{\ms}_{\fp}(K,T)$ and $y^* \in \ch{\ms}_{\fp^*}(K,T^*)$ are both
non-zero.  Then
\begin{align*}
\frac{\exp_{\fp}^{*}(y^*) \cdot
  \exp_{\fp^*}^{*}(y)}{[y,y^*]_{K,\fp^*}} &= 
(p-1)^{-1} \cdot p \cdot \bN(\ff) \cdot \left( \frac{\psi^*(\fp)}{p} -1
  \right) \cdot \left( \frac{\psi(\fp^*)}{p} -1 \right) \\
&\times 
\frac{\cL_{\fp}(\psi)}{(1-\psi^{*-1}(\fp)) \cdot (1
  -\psi^{-1}(\fp^*)) \cdot \Omega_{\fp^*} \cdot \Omega_{\fp} \cdot
  \cL'_{\fp}(\psi^*)}.
\end{align*}
\end{theorem}

\begin{proof}
Since $\cL_{\fp}(\psi) \neq 0$, it follows that
\begin{equation*}
\Hom( \ch{\ms}_{\fp}(K,T) \otimes_{\bZ_p} \ch{\ms}_{\fp^*}(K,T^*),
\bZ_p)
\end{equation*}
is a free $\bZ_p$-module of rank one (see \cite[Lemma 3.6 and
  Proposition 6.7]{A1}), and that both $\exp_{\fp^*}^{*}(\,\,) \cdot
\exp_{\fp}^{*}(\,\,)$ and $[\,,\,]_{K,\fp^*}$ are non-zero elements of
this module. Hence, since $s_{\fp}$ and $s_{\fp^*}$ are of infinite
order, we have that
\begin{equation*}
\frac{\exp_{\fp^*}^{*}(y) \cdot
  \exp_{\fp}^{*}(y^*)}{[y,y^*]_{K,\fp^*}} =
\frac{\exp_{\fp^*}^{*}(s_{\fp}) \cdot
  \exp_{\fp}^{*}(s_{\fp^*})}{[s_{\fp}, s_{\fp^*}]_{K,\fp^*}}.
\end{equation*}
Applying Theorems \ref{T:infinite} and \ref{T:nondeg}, we see that
\begin{align*}
\frac{\exp_{\fp^*}^{*}(s_{\fp}) \cdot
  \exp_{\fp}^{*}(s_{\fp^*})}{[s_{\fp}, s_{\fp^*}]_{K,\fp^*}} &= 
(p-1)^{-1} \cdot p \cdot \bN(\ff) \cdot \left( \frac{\psi^*(\fp)}{p} -1
  \right) \cdot \left( \frac{\psi(\fp^*)}{p} -1 \right) \\
&\times
\frac{\cL_{\fp^*}(\psi^*)}{(1-\psi^{*-1}(\fp)) \cdot (1
  -\psi^{-1}(\fp^*)) \cdot \Omega_{\fp^*} \cdot \Omega_{\fp} \cdot
  \cL'_{\fp}(\psi^*)}.
\end{align*}
The desired result now follows from the fact that
$\cL_{\fp}(\psi)=\cL_{\fp^*}(\psi^*)$ (cf. Remark \ref{R:switch}).
\end{proof}

This completes the proof of Theorem \ref{T:A}.


\newpage

\begin{thebibliography}{99}

\bibitem{A1}
A. Agboola,
\emph{On Rubin's variant of the $p$-adic Birch and Swinnerton-Dyer
  conjecture}, Comp. Math. {\bf 143} (2007), 1374--1398.

\bibitem{A2}
A. Agboola,
\emph{On certain special values of the Katz two-variable $p$-adic
  $L$-function}, in preparation.

\bibitem{BGS} 
D. Bernardi, C. Goldstein, N. Stephens, 
\emph{Notes $p$-adiques sur les courbes elliptiques}, Crelle {\bf 351} (1985),
  129--170.

\bibitem{B}
A. Brumer,
\emph{On the units of algebraic number fields}, Mathematika {\bf 14}
(1967), 121--124.

%\bibitem{Ca}
%J. Cassels,
%\emph{Arithmetic on curves of genus $1$ (VIII). On Conjectures of
%Birch and Swinnerton-Dyer},
%Crelle {\bf 217} (1965), 180--199.

%\bibitem{C}
%J. Coates,
%\emph{Infinite descent on elliptic curves with complex
%  multiplication}, Shaferevich birthday volume.

%\bibitem{CM}
%J. Coates, G. McConnell,
%\emph{Iwasawa theory of modular elliptic curves of analytic rank at
%  most $1$},
%J. London Math. Soc. {\bf 50} (1994), 243--264.

%\bibitem{CS}
%J. Coates, R. Sujatha,
%\emph{Galois cohomology of elliptic curves},
%Narosa Publishing House (2000).

\bibitem{dS}
E. de Shalit,
\emph{Iwasawa theory of elliptic curves with complex multiplication},
Academic Press (1987).

%\bibitem{Gr78}
%R. Greenberg,
%\emph{On the structure of certain Galois groups},
%Invent. Math. {\bf 47} (1978), 85--99.

\bibitem{Gr94}
R. Greenberg,
\emph{Trivial zeros of $p$-adic $L$-functions},
In: $p$-adic monodromy and the Birch and Swinnerton-Dyer conjecture
(Boston, MA, 1991), 149--174, Contemp. Math., 165, Amer. Math. Soc.,
Providence, RI, 1994.

\bibitem{MTT}
B. Mazur, J. Tate, J. Teitelbaum,
\emph{On $p$-adic analogues of the conjectures of Birch and
Swinnerton-Dyer},
Invent. Math. {\bf 84} (1986), 1--48.

\bibitem{PR1}
B. Perrin-Riou,
\emph{D\'escent infinie et hauteurs $p$-adiques sur les courbes
elliptiques \`a multiplication complexe},
Invent. Math. {\bf 70} (1983), 369--398.

\bibitem{PR6} 
B. Perrin-Riou,
\emph{Th\'eorie d'Iwasawa et hauteurs $p$-adiques},
Invent. Math. {\bf 109} (1992), 137--185.

\bibitem{PR7} 
B. Perrin-Riou,
\emph{Arithm\'etique des courbes elliptiques et th\'eorie d'Iwasawa},
Memoire de la Soci\'et\'e Math\'ematique de France, {\bf 17} (1984).

%\bibitem{R85}
%K. Rubin,
%\emph{Elliptic curves and $\bZ_p$-extensions},
%Comp. Math. {\bf 56} (1985) 237--250.

\bibitem{Ru}
K. Rubin,
\emph{Tate-Shafarevich groups and $L$-functions of elliptic curves
with complex multiplication}, 
Invent. Math. {\bf 89} (1987), 527--560.

\bibitem{R}
K. Rubin,
\emph{$p$-adic $L$-functions and rational points on elliptic curves
with complex multiplication},
Invent. Math. {\bf 107} (1992), 323--350.

\bibitem{R1}
K. Rubin,
\emph{$p$-adic variants of the Birch and Swinnerton-Dyer conjecture
for elliptic curves with complex multiplication}.
In: $p$-adic monodromy and the Birch and Swinnerton-Dyer conjecture
(Boston, MA, 1991), 71--80, Contemp. Math., 165, Amer. Math. Soc.,
Providence, RI, 1994.

\bibitem{R2}
K. Rubin,
\emph{The ``main conjectures'' of Iwasawa theory for imaginary
quadratic fields}, Invent. Math. {\bf 109} (1992), 25--68.

\bibitem{R3} 
K. Rubin,
\emph{Euler Systems}, Princeton University Press (2000).

\bibitem{R4}
K. Rubin,
\emph{Euler systems and modular elliptic curves}.
In: Galois representations in arithmetic geometry, A.J. Scholl,
R. L. Taylor (Eds.), CUP 1998.

\bibitem{S1}
J.-P. Serre,
\emph{Local Fields}, Springer Verlag (1979).

\end{thebibliography}
\end{document}